\documentclass[12pt,reqno]{amsart}
\usepackage{amsmath,amssymb,amsthm,mathtools,enumerate,url}
\usepackage{framed}
\allowdisplaybreaks
\DeclareFontFamily{U}{mathx}{\hyphenchar\font45}
\DeclareFontShape{U}{mathx}{m}{n}{
      <5> <6> <7> <8> <9> <10>
      <10.95> <12> <14.4> <17.28> <20.74> <24.88>
      mathx10
      }{}
\DeclareSymbolFont{mathx}{U}{mathx}{m}{n}
\DeclareFontSubstitution{U}{mathx}{m}{n}
\DeclareMathAccent{\widecheck}{0}{mathx}{"71}
\DeclareFontEncoding{OT2}{}{}
\DeclareFontSubstitution{OT2}{cmr}{m}{n}
\DeclareFontFamily{OT2}{cmr}{\hyphenchar\font45 }
\DeclareFontShape{OT2}{cmr}{m}{n}{
<5><6><7><8><9>gen*wncyr
<10><10.95><12><14.4><17.28><20.74><24.88>wncyr10}{}
\DeclareFontShape{OT2}{cmr}{b}{n}{
<5><6><7><8><9>gen*wncyb
<10><10.95><12><14.4><17.28><20.74><24.88>wncyb10}{}
\DeclareMathAlphabet{\mathcyr}{OT2}{cmr}{m}{n}
\DeclareMathAlphabet{\mathcyb}{OT2}{cmr}{b}{n}
\SetMathAlphabet{\mathcyr}{bold}{OT2}{cmr}{b}{n}
\calclayout
\theoremstyle{plain}
\newtheorem{theorem}{Theorem}[section]
\newtheorem{proposition}[theorem]{Proposition}
\newtheorem{lemma}[theorem]{Lemma}
\newtheorem{corollary}[theorem]{Corollary}
\theoremstyle{definition}
\newtheorem{definition}[theorem]{Definition}
\newtheorem{example}[theorem]{Example}
\newtheorem{remark}[theorem]{Remark}
\newcommand{\jump}[1]{\ensuremath{[\![#1]\!]}}
\newcommand{\emp}{\varnothing}
\newcommand{\BB}{\mathbb{B}}
\newcommand{\CC}{\mathbb{C}}
\newcommand{\DD}{\mathbb{D}}
\newcommand{\ee}{\boldsymbol{\epsilon}}
\newcommand{\ba}{\boldsymbol{a}}
\newcommand{\bk}{\boldsymbol{k}}
\newcommand{\bl}{\boldsymbol{l}}

\newcommand{\bm}{\boldsymbol{m}}

\newcommand{\bq}{\boldsymbol{q}}
\newcommand{\QQ}{\mathbb{Q}}
\newcommand{\RR}{\mathbb{R}}
\newcommand{\bw}{\boldsymbol{w}}
\newcommand{\ZZ}{\mathbb{Z}}
\newcommand{\bz}{\boldsymbol{z}}
\newcommand{\fZ}{\mathfrak{Z}}
\newcommand{\Li}{\mathrm{Li}^{\sh}}
\newcommand{\dep}{\mathrm{dep}}
\newcommand{\wt}{\mathrm{wt}}
\newcommand{\up}{\uparrow}

\newcommand{\hof}{\mathcal{A}}
\newcommand{\sh}{\mathbin{\mathcyr{sh}}}
\numberwithin{equation}{section}
\address[Hanamichi Kawamura]{Department of Mathematics, Faculty of Science Division I, Tokyo University of Science, 1-3 Kagurazaka, Shinjuku-ku, Tokyo, 162-8601, Japan}
\email{1121026@ed.tus.ac.jp}
\address[Takumi Maesaka]{Department of Information and Computer Science, Kanazawa Institute of Technology, 7-1 Ohgigaoka, Nonoichi-shi, Ishikawa, 921-8501, Japan}
\email{c1117334@planet.kanazawa-it.ac.jp}
\address[Shin-ichiro Seki]{Department of Mathematical Sciences\\ Aoyama Gakuin University\\ 5-10-1 Fuchinobe, Chuo-ku, Sagamihara-shi, Kanagawa, 252-5258, Japan}
\email{seki@math.aoyama.ac.jp}
\title{Multivariable connected sums and multiple polylogarithms}
\author{Hanamichi Kawamura, Takumi Maesaka, Shin-ichiro Seki}
\date{}
\thanks{This research was supported in part by JST Global Science Campus ROOT program and by JSPS KAKENHI Grant Number 18J00151, 21K13762.}
\begin{document}
\maketitle
\begin{abstract}
We introduce the multivariable connected sum which is a generalization of Seki--Yamamoto's connected sum and prove the fundamental identity for these sums by series manipulation.
This identity yields explicit procedures for evaluating multivariable connected sums and for giving relations among special values of multiple polylogarithms.
In particular, our class of relations contains Ohno's relations for multiple polylogarithms.
\end{abstract}

\section{Introduction}\label{sec:Intro}
In \cite{Weisstein1}, a symmetrical double series expression of $\zeta(2)$ due to B. Cloitre is exhibited:\begin{equation}\label{eq:Cloitre}
\zeta(2)=\sum_{m=1}^{\infty}\sum_{n=1}^{\infty}\frac{(m-1)!(n-1)!}{(m+n)!}.
\end{equation}
The right-hand side of \eqref{eq:Cloitre} is a special case of the \emph{connected sum} introduced by the third author and Yamamoto \cite{Seki-Yamamoto}.
We provide a brief review of this theory here.
We call a tuple of positive integers $\bk=(k_1,\dots,k_r)$ an \emph{index}.
If $k_r\geq 2$, then we call $\bk$ an \emph{admissible} index.
We call $r$ (resp.~$k_1+\cdots+k_r$) the \emph{depth} (resp.~\emph{weight}) of $\bk$ and denote it by $\dep(\bk)$ (resp.~$\wt(\bk)$).
We regard the empty tuple $\emp$ as an admissible index and call it the \emph{empty index}.
We set $\dep(\emp)\coloneqq 0$ and $\wt(\emp)\coloneqq 0$.
For two indices $\bk=(k_1,\dots,k_r)$ and $\bl=(l_1,\dots, l_s)$, the connected sum $Z(\bk;\bl)$ is defined by
\[
Z(\bk;\bl)\coloneqq\sum_{\substack{0=m_0<m_1<\cdots<m_r \\ 0=n_0<n_1<\cdots<n_s}}\frac{m_r!n_s!}{(m_r+n_s)!}\cdot\left[\prod_{i=1}^r\frac{1}{m_i^{k_i}}\right]\cdot\left[\prod_{j=1}^s\frac{1}{n_j^{l_j}}\right]\in\RR\cup\{+\infty\}.
\]
Throughout, we understand that the empty product (resp.~sum) is $1$ (resp.~$0$).
It is straightforward to see that the symmetry $Z(\bk;\bl)=Z(\bl;\bk)$ and the \emph{boundary condition} $Z(\bk;\emp)=\zeta(\bk)$ hold.
Here $\zeta(\bk)$ is the \emph{multiple zeta value} (MZV)
\[
\zeta(\bk)\coloneqq\sum_{0=m_0<m_1<\cdots<m_r}\frac{1}{m_1^{k_1}\cdots m_r^{k_r}}
\]
and this series is convergent if $\bk$ is admissible.
The most important properties for the connected sum are the \emph{transport relations}.
To state them, we introduce \emph{arrow notation} for indices.
For a non-empty index $\bk=(k_1,\dots,k_r)$, we define $\bk_{\to}$ (resp.~$\bk_{\up}$) to be $(k_1,\dots,k_r,1)$ (resp.~$(k_1,\dots,k_{r-1},k_r+1)$).
We set $\emp_{\to}\coloneqq(1)$ and $\emp_{\up}\coloneqq\emp$.
Then the transport relations for $Z(\bk;\bl)$ are as follows (the $q\to1$, $x\to0$ case of \cite[Theorem~2.2]{Seki-Yamamoto}):
\begin{equation}\label{eq:SY-trans}
\begin{split}
Z(\bk_{\to};\bl)&=Z(\bk;\bl_{\up}) \quad (\bl\neq\emp),\\
Z(\bk_{\up};\bl)&=Z(\bk;\bl_{\to}) \quad (\bk\neq\emp).
\end{split}
\end{equation}
By using these properties of the connected sum, we can prove the duality relation for MZVs ``$\zeta(\bk)=\zeta(\bk^{\dagger})$.''
Here, we explain the definition of the dual index $\bk^{\dagger}$.
For a non-negative integer $m$ and a complex number $z$, the symbol $\{z\}^m$ denotes $m$ repetitions of $z$.
Any admissible index $\bk\neq\emp$ is uniquely expressed as $\bk=(\{1\}^{a_1-1},b_1+1,\dots,\{1\}^{a_h-1},b_h+1)$, where $h,a_1,\dots,a_h,b_1,\dots,b_h$ are positive integers.
Then, the dual index is defined as $\bk^{\dagger}\coloneqq(\{1\}^{b_h-1},a_h+1,\dots,\{1\}^{b_1-1},a_1+1)$.
For the empty index, we use $\emp^{\dagger}\coloneqq\emp$.
The duality relation is immediately proved by the change of variables in the iterated integral representation of the MZV (a special case of \eqref{eq:Goncharov} below).
We emphasize that Seki--Yamamoto's proof is executed using only series manipulation and no iterated integral is required.

Since the single index $\bk=(2)$ is self-dual, the duality for this $\bk$ is a tautology ``$\zeta(2)=\zeta(2)$.''
However, the transportation process includes a non-trivial component as follows:
\begin{equation}\label{eq:Tautological}
\zeta(2)=Z(2;\emp)=Z(1;1)=Z(\emp;2)=\zeta(2).
\end{equation}
We see that the identity $\zeta(2)=Z(1;1)$ is nothing but Cloitre's identity \eqref{eq:Cloitre}.
Note that parentheses will often be omitted, such as $Z((2);\emp)=Z(2;\emp)$ and $Z((1);(1))=Z(1;1)$.

In \cite{Weisstein2}, an elegant series expression of Ap\'ery's constant $\zeta(3)$ owing to O. Oloa is exhibited:
\begin{equation}\label{eq:Oloa}
\zeta(3)=\frac{1}{3}\sum_{m=1}^{\infty}\sum_{n=1}^{\infty}\frac{(m-1)!(n-1)!}{(m+n)!}\sum_{j=1}^{m+n}\frac{1}{j}.
\end{equation}
Identities \eqref{eq:Cloitre} and \eqref{eq:Oloa} are also discussed on the website ``les-mathematiques.net'' in a forum post entitled ``Acc\'el\'eration pour zeta(2)'' \cite{LM1}.
In a post entitled ``S\'erie triple'' \cite{LM2} on the same website, the following series identity is given:
\[
\sum_{m_1=1}^{\infty}\sum_{m_2=1}^{\infty}\sum_{m_3=1}^{\infty}\frac{(m_1-1)!(m_2-1)!(m_3-1)!}{(m_1+m_2+m_3)!}=\frac{13}{4}\zeta(3)-\frac{\pi^2}{2}\log 2.
\]
Furthermore, a user named Amtagpa showed
\begin{equation}\label{eq:Amtagpa}
\sum_{m_1=1}^{\infty}\cdots\sum_{m_{n+1}=1}^{\infty}\frac{(m_1-1)!\cdots(m_{n+1}-1)!}{(m_1+\cdots+m_{n+1})!}=n!\Li_{\{1\}^{n-1},2}\left(1,\frac{1}{2},\dots,\frac{1}{n}\right)
\end{equation}
for a general positive integer $n$ by using the iterated integral representation.
Here the \emph{multiple polylogarithm} (MPL) $\Li_{\bk}(z_1,\dots, z_r)$ is of the so-called shuffle-type:
\[
\Li_{\bk}(z_1,\dots,z_r)\coloneqq\sum_{0=m_0<m_1<\cdots<m_r}\prod_{i=1}^r\frac{z_i^{m_i-m_{i-1}}}{m_i^{k_i}},
\]
where $\bk=(k_1,\dots, k_r)$ is an index and $z_1,\dots, z_r$ are complex numbers satisfying $|z_1|$, $\dots,|z_r|\leq 1$.
If $(k_r,|z_r|)\neq (1,1)$, then this series is absolutely convergent.
The shuffle-type MPL has the iterated integral representation
\begin{equation}\label{eq:Goncharov}
\Li_{\bk}(z_1,\dots,z_r)=(-1)^r\mathrm{I}_{\bk}(z_1^{-1},\dots,z_r^{-1})
\end{equation}
in the notation of \cite{G}.
The shuffle-type MPLs inherit the shuffle product property of iterated integrals directly in the $z_i$ variables and this fact is the reason why we call these series ``shuffle-type.''
There is another type of MPL, the harmonic-type MPL used and defined in \S\ref{sec:boundary condtions}, but we principally focus on the shuffle-type MPL in this paper.
We can check that $2\Li_{1,2}\left(1,\frac{1}{2}\right)=\frac{13}{4}\zeta(3)-\frac{\pi^2}{2}\log 2$ holds.
For example, by the duality for MPLs (the $n=2$ case of \eqref{eq:11112duality}), we see that $\Li_{1,2}\left(1,\frac{1}{2}\right)=-\zeta(\overline{1},\overline{2})$; see \S\ref{subsec:duality for MPL} for the notation for the alternating multiple zeta values.
Then the evaluation $-2\zeta(\overline{1},\overline{2})=\frac{13}{4}\zeta(3)-\frac{\pi^2}{2}\log 2$ is confirmed by \cite[Proposition~14.2.7]{Zhao}.

Our focus in this paper is on interpreting identities \eqref{eq:Oloa} and \eqref{eq:Amtagpa} by transporting indices in the same way as in Cloitre's identity \eqref{eq:Cloitre}.
For this purpose, we extend Seki--Yamamoto's connected sum to the \emph{multiple connected sum} as follows.
We set $S_r\coloneqq\{\bm=(m_0,m_1,\dots,m_r)\in\ZZ^{r+1}\mid m_0=0<m_1<\cdots<m_r\}$ and $S^{\star}_s(m)\coloneqq\{\bq=(q_0,q_1,\dots,q_s)\in\ZZ^{s+1}\mid q_0=0<q_1\leq\cdots\leq q_s=m\}$ for non-negative integers $r,s$ and $m$.
Furthermore, we set
\[
C(a_1,a_2,\dots,a_n)\coloneqq\frac{(a_1)!(a_2)!\cdots(a_n)!}{(a_1+a_2+\cdots+a_n)!}
\]
for a positive integer $n$ and non-negative integers $a_1,\dots, a_n$.
This value $C(a_1,a_2,\dots,a_n)$ is just the inverse of a multinomial coefficient, but in the context of multiple connected sums, we call it the \emph{connector} because of its role.
Using these notation, for a positive integer $n$ and indices $\bk_1=(k_1^{(1)},\dots,k_{r_1}^{(1)}),\dots, \bk_n=(k_1^{(n)},\dots,k_{r_n}^{(n)})$ and $\bl=(l_1,\dots,l_s)$,  we define the multiple connected sum $Z_n(\bk_1;\dots;\bk_n\mid\bl)$ as
\begin{align*}
&Z_n(\bk_1;\dots;\bk_n\mid\bl)\\
&\coloneqq\sum_{\substack{\bm^{(1)}\in S_{r_1}, \dots, \bm^{(n)}\in S_{r_n} \\ \bq\in S^{\star}_s(m_{r_1}^{(1)}+m_{r_2}^{(2)}+\cdots+m_{r_n}^{(n)})}}C(m_{r_1}^{(1)},m_{r_2}^{(2)},\dots,m_{r_n}^{(n)})\cdot\left[\prod_{j=1}^n\left(\prod_{i=1}^{r_j}\frac{1}{(m_i^{(j)})^{k_i^{(j)}}}\right)\right]\cdot\left[q_s\prod_{i=1}^s\frac{1}{q_i^{l_i}}\right],
\end{align*}
where $\bm^{(1)}=(m_1^{(1)},\dots,m_{r_1}^{(1)}),\dots, \bm^{(n)}=(m_1^{(n)},\dots, m_{r_n}^{(n)})$ and $\bq=(q_1,\dots, q_s)$.
Set $Z_n(\bk_1;\dots;\bk_n\mid\emp)\coloneqq0$ and $Z_n(\bk_1;\dots;\bk_n)\coloneqq Z_n(\bk_1;\dots;\bk_n\mid1)$.

The original Seki--Yamamoto's connected sum $Z(\bk;\bl)$ coincides with our $Z_2(\bk;\bl)=Z_2(\bk;\bl\mid1)$, not $Z_1(\bk\mid \bl)$.
The series appearing in identities \eqref{eq:Oloa} and \eqref{eq:Amtagpa} are $Z_2(1;1\mid1,1)$ and $Z_{n+1}(1;\dots;1)$, respectively.
The presence of the $\bl$-component in the definition is advantageous for obtaining relations among MZVs or MPLs, and the $n=1$ case of our connected sum has been explored in previous studies; $Z_1(\bk\mid\bl_{\up})$ ($\bk, \bl\neq\emp$) is the multiple zeta value of Kaneko--Yamamoto type introduced in \cite{Kaneko-Yamamoto} and $Z_1(\bk\mid \bl_{\to})$ ($\bk$ is admissible) is the Schur multiple zeta value of anti-hook type introduced in \cite{Nakasuji-Phuksuwan-Yamasaki}.
Note that these two kinds of series are the same object by definition or by the $n=1$ case of Proposition~\ref{prop:tate_transport} below.

In what follows, we provide an explicit procedure for calculating $Z_n(\bk_1;\dots;\bk_n\mid\bl)$ via transportations of indices and prove the transport relations by series manipulation.
In addition to identities \eqref{eq:Oloa} and \eqref{eq:Amtagpa}, our method leads to a formula for $\zeta(4)$ (see Example~\ref{ex:new formula for zeta(4)}):
\begin{equation}\label{eq:Oloa-like}
\zeta(4)=\frac{8}{17}\sum_{m=1}^{\infty}\sum_{n=1}^{\infty}\frac{(m-1)!(n-1)!}{(m+n)!}\sum_{j=1}^{m+n}\frac{1}{j^2}.
\end{equation}
More generally, we prove the following.
\begin{theorem}\label{thm:main_with_no_variables}
Let $n$ be an integer at least $2$ and $\bk_1,\ldots,\bk_n,\bl$ non-empty indices.
We set $k\coloneqq \wt(\bk_1)+\cdots+\wt(\bk_n)+\wt(\bl)$ and define $\mathrm{MPL}(n,k)$ by
\[
\mathrm{MPL}(n,k)\coloneqq\left\{\Li_{\bk}(1,z_2,\dots, z_r) \ \middle| \ {1\leq r\leq k-2, \bk \ \text{is admissible with } \wt(\bk)=k-1, \atop \dep(\bk)=r, \text{and } z_2,\dots,z_r\in \{1,\frac{1}{2},\dots,\frac{1}{n}\}}\right\}.
\]
Then, the multiple connected sum $Z_n(\bk_1;\dots;\bk_n\mid\bl)$ can be written explicitly as a $\ZZ$-linear combination of elements of $\mathrm{MPL}(n,k)$.
\end{theorem}
The proof of this theorem by transporting indices requires the use of more general connected sums; see Example~\ref{ex:tranport relation for Z_3} for instance.
In view of this phenomenon, it is natural to define the following generalization of the multiple connected sum.
Let $n$ be a positive integer and $\bk_1=(k_1^{(1)},\dots,k_{r_1}^{(1)}),\dots, \bk_n=(k_1^{(n)},\dots,k_{r_n}^{(n)})$ and $\bl=(l_1,\dots,l_s)$ indices.
Let $\bz_1=(z_1^{(1)},\dots,z_{r_1}^{(1)}),\dots, \bz_n=(z_1^{(n)},\dots,z_{r_n}^{(n)})$ and $\bw=(w_1,\dots, w_s)$ be elements of $\DD^{r_1},\dots,\DD^{r_n}$ and $\DD^{s}$, respectively.
Here, the symbol $\DD$ denotes the unit disk $\{z\in\CC\mid |z|\leq 1\}$ and the unique element of $\DD^0$ is denoted by $\varnothing$.
Then, the \emph{multivariable connected sum} is defined by
\begin{multline*}
Z_n\left({\bz_1 \atop \bk_1};\dots;{\bz_n \atop \bk_n} \ \middle| \ {\bw \atop \bl}\right)\coloneqq\sum_{\substack{\bm^{(1)}\in S_{r_1}, \dots, \bm^{(n)}\in S_{r_n} \\ \bq\in S^{\star}_s(m_{r_1}^{(1)}+m_{r_2}^{(2)}+\cdots+m_{r_n}^{(n)})}}C(m_{r_1}^{(1)},m_{r_2}^{(2)},\dots,m_{r_n}^{(n)})\\
\cdot\left[\prod_{j=1}^n\left(\prod_{i=1}^{r_j}\frac{(z_i^{(j)})^{m_i^{(j)}-m_{i-1}^{(j)}}}{(m_i^{(j)})^{k_i^{(j)}}}\right)\right]\cdot\left[q_s\prod_{i=1}^s\frac{w_i^{q_i-q_{i-1}}}{q_i^{l_i}}\right],
\end{multline*}
where $\bm^{(1)}=(m_1^{(1)},\dots,m_{r_1}^{(1)}),\dots, \bm^{(n)}=(m_1^{(n)},\dots, m_{r_n}^{(n)})$ and $\bq=(q_1,\dots, q_s)$.
Set $Z_n\left({\bz_1\atop \bk_1};\dots;{\bz_n\atop\bk_n} \ \middle| \ {\emp\atop\emp}\right)\coloneqq0$ and $Z_n\left({\bz_1 \atop \bk_1};\dots;{\bz_n \atop \bk_n}\right)\coloneqq Z_n\left({\bz_1 \atop \bk_1};\dots;{\bz_n \atop \bk_n} \ \middle| \ {1 \atop 1}\right)$.
We establish the \emph{fundamental identity} for multivariable connected sums.
To do so, we generalize the arrow notation appropriately.
Let $\bk$ be an index and $\bz\in\DD^{\dep(\bk)}$.
Let $v$ be an element of $\DD\setminus\{0\}$.
Then, we use the following notation:
\[
\left({\bz \atop \bk}\right)_{\overset{v}{\longrightarrow}}\coloneqq \left({(\bz,v) \atop \bk_{\to}}\right), \quad \left({\bz \atop \bk}\right)_{\overset{0}{\longrightarrow}}\coloneqq\left({\bz \atop \bk}\right)_{\up}\coloneqq\left({\bz \atop \bk_{\up}}\right),\quad \left({\bz \atop \bk}\right)_{\overset{\infty}{\longrightarrow}}\coloneqq -\left({\bz \atop \bk}\right)_{\up}.
\]
In relation to the last symbol, we use the following definition:
\[
Z_n\left(\epsilon_1\left({\bz_1 \atop \bk_1}\right);\dots;\epsilon_n\left({\bz_n \atop \bk_n}\right) \ \middle| \ {\bw \atop \bl}\right)\coloneqq \epsilon_1\cdots\epsilon_n\cdot Z_n\left({\bz_1 \atop \bk_1};\dots;{\bz_n \atop \bk_n} \ \middle| \ {\bw \atop \bl}\right)
\]
for $\epsilon_1,\dots, \epsilon_n\in\{1,-1\}$.

The following identity is key to this paper.
\begin{theorem}[Fundamental identity]\label{thm:fund_id}
Let $n$ be an integer greater than $1$ and $\bk_1,\dots, \bk_{n+1}$ indices. 
Take tuples of complex numbers $\bz_i\in\DD^{\dep(\bk_i)}$ for all $1\leq i \leq n+1$.
Furthermore, take $v_1,\dots, v_n\in(\DD\setminus\{0\})\cup\{\infty\}$ and $v_{n+1}\in\DD$ satisfying
\[
\frac{1}{v_1}+\cdots+\frac{1}{v_n}=v_{n+1}
\]
and the assumption that 
\begin{enumerate}[$(1)$]
\item if $\bk_i=\emp$ $(1\leq i\leq n)$, then $v_i\neq \infty$, and
\item if $n=2$ and $\bk_i=\emp$ for $1\leq i\leq 2$, then $|v_{3-i}|\neq 1$ or $|v_3|<1$.
\end{enumerate}
Here, we understand $1/\infty=0$ and $|\infty|=\infty$.
Then, we have

\begin{framed}
\[
\sum_{i=1}^{n+1}Z_n\left(\left({\bz_1 \atop \bk_1}\right)_{\overset{v_1}{\longrightarrow}};\dots;{\bz_i\atop \bk_i};\dots;\left({\bz_n \atop \bk_n}\right)_{\overset{v_n}{\longrightarrow}} \ \middle| \ \left({\bz_{n+1} \atop \bk_{n+1}}\right)_{\overset{v_{n+1}}{\longrightarrow}}\right)=0.
\]
\end{framed}

\noindent Here, the symbol $Z_n\left(\left({\bz_1 \atop \bk_1}\right)_{\overset{v_1}{\longrightarrow}};\dots;{\bz_i\atop \bk_i};\dots;\left({\bz_n \atop \bk_n}\right)_{\overset{v_n}{\longrightarrow}} \ \middle| \ \left({\bz_{n+1} \atop \bk_{n+1}}\right)_{\overset{v_{n+1}}{\longrightarrow}}\right)$ means that only the $i$th component is not equipped with the arrow notation.
\end{theorem}
For example, the $n=3$ case for the fundamental identity is
\begin{align*}
&Z_3\left({\bz_1 \atop \bk_1};\left({\bz_2 \atop \bk_2}\right)_{\overset{v_2}{\longrightarrow}};\left({\bz_{3} \atop \bk_{3}}\right)_{\overset{v_{3}}{\longrightarrow}} \ \middle| \ \left({\bz_{4} \atop \bk_{4}}\right)_{\overset{v_{4}}{\longrightarrow}}\right) \\
&+Z_3\left(\left({\bz_1 \atop \bk_1}\right)_{\overset{v_1}{\longrightarrow}};{\bz_2 \atop \bk_2};\left({\bz_{3} \atop \bk_{3}}\right)_{\overset{v_{3}}{\longrightarrow}} \ \middle| \ \left({\bz_{4} \atop \bk_{4}}\right)_{\overset{v_{4}}{\longrightarrow}}\right) \\
&+Z_3\left(\left({\bz_1 \atop \bk_1}\right)_{\overset{v_1}{\longrightarrow}};\left({\bz_2 \atop \bk_2}\right)_{\overset{v_2}{\longrightarrow}};{\bz_{3} \atop \bk_{3}} \ \middle| \ \left({\bz_{4} \atop \bk_{4}}\right)_{\overset{v_{4}}{\longrightarrow}}\right) \\
&+Z_3\left(\left({\bz_1 \atop \bk_1}\right)_{\overset{v_1}{\longrightarrow}};\left({\bz_2 \atop \bk_2}\right)_{\overset{v_2}{\longrightarrow}};\left({\bz_{3} \atop \bk_{3}}\right)_{\overset{v_{3}}{\longrightarrow}} \ \middle| \ {\bz_{4} \atop \bk_{4}}\right) \\
&=0.
\end{align*}
The fundamental identity supplies \emph{all} needed transport relations for the multivariable connected sums in this paper.
By transporting indices and variables algorithmically, in accordance with the method presented in \cite{Seki}, a given multivariable connected sum satisfying a suitable condition is written as a $\ZZ$-linear combination of $Z_1$-values.
We regard the boundary conditions for multivariable connected sums as the fact that $Z_1$-values are written as $\ZZ$-linear combinations of MPLs; if $\bw=\bl=(1)$, then the boundary conditions are trivial.

Thus, main theorems of the present paper are  Theorem~\ref{thm:main_with_no_variables}, Theorem~\ref{thm:fund_id} and the following two theorems:
\begin{theorem}[see Theorem~\ref{thm:main1 precise} for the precise statement]\label{thm:main1}
Every multivariable connected sum with transportable variables can be expressed explicitly as a $\ZZ$-linear combination of absolutely convergent multiple polylogarithms.
\end{theorem}
\begin{theorem}[see \S\ref{subsec:Recipe for relations} for the precise recipe]\label{thm:main2}
Transport relations and boundary conditions for multivariable connected sums give a family of functional relations among multiple polylogarithms.
As a special case, Ohno's relation for multiple polylogarithms is obtained.
\end{theorem}
For simplicity, conditional convergent series and analytic continuations are disregarded herein.

The remainder of the paper is organized as follows.
In \S\ref{sec:the fundamental identity}, we prove the fundamental identity.
In \S\ref{sec:transport relations}, we explain transport relations derived from the fundamental identity and reprove the duality relation for MPLs.
In \S\ref{sec:boundary condtions}, we describe the boundary conditions.
In \S\ref{sec:Evaluations}, we give the recipe for obtaining evaluations of multivariable connected sums, and prove Theorem~\ref{thm:main_with_no_variables} and Theorem~\ref{thm:main1}.
Finally, in \S\ref{sec:Relations}, we give the recipe for obtaining relations among MPLs and prove Theorem~\ref{thm:main2}.
Ohno's relation for MPLs is also proved here.
Some examples are exhibited in \S\ref{sec:transport relations}--\S\ref{sec:Relations}.

\section*{acknowledgement}
The authors would like to express their sincere gratitude to Professor Masanobu Kaneko and Professor Tetsushi Ito for their helpful comments.
The first author would like to thank Dr.~Junpei Tsuji for guiding him through the program code for experiments on our proof of Ohno's relation.
The authors are grateful to the anonymous referee for valuable comments and useful suggestions to improve the manuscript. 
\section{The fundamental identity}\label{sec:the fundamental identity}
First we prepare basic properties of the multivariable connected sum.
\begin{proposition}\label{prop:fundamental_properties}
Let $n$ be a positive integer and $\bk_1,\dots, \bk_n, \bl$ indices.
Take tuples of complex numbers $\bz_i\in\DD^{\dep(\bk_i)}$ for all $1\leq i\leq n$ and $\bw\in\DD^{\dep(\bl)}$.
Then the following statements hold:
\begin{enumerate}[\rm(i)]
\item\label{it:2.1-1} For $\bz_1=(z_1,\dots, z_r), \bw=(w_1,\dots, w_s)$ $(r=\dep(\bk_1), s=\dep(\bl))$, we assume that $|z_r|<1$ or $|w_s|<1$ holds if both $\bk_1$ and $\bl$ are non-admissible.
Then $Z_1\left({\bz_1\atop\bk_1} \ \middle| \ {\bw \atop \bl}\right)$ converges absolutely, and for the case of $\bl=(1)$, $\bw=(1)$, we have $Z_1\bigl({\bz_1\atop\bk_1}\bigr)=\Li_{\bk_1}(\bz_1)$.
\item\label{item:conv_n} If $n>1$ and all $\bk_i$ are non-empty, then the multivariable connected sum
\[
Z_n\left({\bz_1 \atop \bk_1};\dots;{\bz_n \atop \bk_n} \ \middle| \ {\bw \atop \bl}\right)
\]
converges absolutely.
\item\label{item:symmetry}
For any permutation  $\sigma$ of degree $n$, 
\[
Z_n\left({\bz_1 \atop \bk_1};\dots;{\bz_n \atop \bk_n} \ \middle| \ {\bw \atop \bl}\right)=Z_n\left({\bz_{\sigma(1)} \atop \bk_{\sigma(1)}};\dots;{\bz_{\sigma(n)} \atop \bk_{\sigma(n)}} \ \middle| \ {\bw \atop \bl}\right)
\]
holds.
\item\label{item:reduction}
Let $n>1$. For all $1\leq i\leq n$, we have
\[
Z_n\left({\bz_1 \atop \bk_1};\dots;{\bz_{i-1}\atop\bk_{i-1}};{\emp \atop \emp};{\bz_{i+1}\atop\bk_{i+1}};\dots;{\bz_n \atop \bk_n} \ \middle| \ {\bw \atop \bl}\right)=Z_{n-1}\left({\bz_1 \atop \bk_1};\dots;{\widecheck{\bz_i} \atop \widecheck{\bk_i}};\dots;{\bz_n \atop \bk_n} \ \middle| \ {\bw \atop \bl}\right),
\]
where ${\widecheck{\bz_i} \atop \widecheck{\bk_i}}$ denotes the skipped entry.
\end{enumerate}
\end{proposition}
\begin{proof}
We prove only \eqref{item:conv_n} because the others are obvious by definition.
It is sufficient to show that $Z_n(\bk_1;\dots;\bk_n\mid \bl)$ converges.
Since $n>1$, the inequality 
\begin{equation}\label{eq:connector ineq}
C(m_1,\dots,m_n)\leq \frac{1}{m_1\cdots m_n}
\end{equation}
holds for positive integers $m_1,\dots, m_n$; the case $n=2$ is proved from the estimate
\[
\binom{m_1+m_2}{m_2}\geq \binom{m_1+m_2}{2}=\binom{m_1}{2}+\binom{m_2}{2}+m_1m_2>m_1m_2\quad \text{for }m_1\geq m_2\geq 2,
\]
and then the general case follows from induction on $n$.
Take any $\varepsilon>0$ and set $s=\dep(\bl)$ and $\bl=(l_1,\dots,l_s)$. Then there exists some positive constant $c_{\varepsilon,s,n}$ depending on $\varepsilon, s, n$ such that
\[
\sum_{1\leq q_1\leq \cdots \leq q_s=m_1+\cdots+m_n}\frac{q_s}{q_1^{l_1}\cdots q_s^{l_s}}\leq c_{\varepsilon,s,n}\cdot (m_1\cdots m_n)^{\varepsilon}.
\]
By these elementary facts, we see that $Z_n(\bk_1;\dots;\bk_n\mid \bl)$ is bounded above by a constant multiple of a product of $n$ convergent special values  of the multiple zeta function.
\end{proof}
By virtue of Proposition~\ref{prop:fundamental_properties} \eqref{item:symmetry}, Theorem~\ref{thm:fund_id} is a consequence of Proposition~\ref{prop:tate_transport} and Theorem~\ref{thm:main_transport_relations} below.
\begin{proposition}\label{prop:tate_transport}
Let $n$ be a positive integer.
Let $\bk_1,\dots, \bk_n, \bl$ be non-empty indices. 
Take tuples of complex numbers $\bz_i\in\DD^{\dep(\bk_i)}$ for all $1\leq i \leq n$ and $\bw\in\DD^{\dep(\bl)}$.
Then we have
\[
\sum_{i=1}^{n}Z_n\left(\left({\bz_{1} \atop \bk_{1}}\right)_{\up};\dots;{\bz_i\atop\bk_i};\dots;\left({\bz_{n} \atop \bk_{n}}\right)_{\up} \ \middle| \ \left({\bw \atop \bl}\right)_{\up}\right)=Z_n\left(\left({\bz_1\atop \bk_1}\right)_{\up};\dots;\left({\bz_n\atop \bk_n}\right)_{\up} \ \middle| \ {\bw\atop\bl}\right).
\]
\end{proposition}
\begin{proof}
This follows from the identity
\[
\sum_{i=1}^n\frac{m_i}{m_1\cdots m_n}=\frac{m_1+\cdots+m_n}{m_1\cdots m_n}
\]
for positive integers $m_1,\dots, m_n$.
\end{proof}
For a convergence argument, we need the following lemma.
\begin{lemma}\label{lem:easy-bound}
Let $n$ and $N$ be integers satisfying $N\geq n\geq2$.
Then we have
\[
\sum_{\substack{(a_1,\dots,a_n)\in\ZZ_{>0}^n\\a_1+\cdots+a_n=N}}\frac{C(a_1,\dots,a_n)}{a_1\cdots a_n}=O\left(\frac{1}{N^2}\right).
\] 
\end{lemma}
\begin{proof}
By estimates $\frac{1}{ab}\leq \frac{2}{a+b}=\frac{2}{N}$ and
\[
\sum_{\substack{(a,b)\in\ZZ_{>0}^2\\a+b=N}}C(a,b)=\sum_{a=1}^{N-1}\frac{1}{\binom{N}{a}}\leq \frac{2}{N}+\sum_{a=2}^{N-2}\frac{1}{\binom{N}{2}}=\frac{2}{N}+\frac{2(N-3)}{N(N-1)}\leq \frac{4}{N},
\]
the statement for $n=2$ is true.
Next, we assume that $n\geq 3$ and the desired estimate for $n-1$ holds.
Then there exists some positive constant $c$ such that we have
\begin{align*}
\sum_{\substack{(a_1,\dots,a_n)\in\ZZ_{>0}^n\\a_1+\cdots+a_n=N}}\frac{C(a_1,\dots,a_n)}{a_1\cdots a_n}&=\sum_{a_n=1}^{N-n+1}\frac{C(N-a_n,a_n)}{a_n}\sum_{\substack{(a_1,\dots,a_{n-1})\in\ZZ_{>0}^{n-1}\\a_1+\cdots+a_{n-1}=N-a_n}}\frac{C(a_1,\dots,a_{n-1})}{a_1\cdots a_{n-1}}\\
&\leq\sum_{a_n=1}^{N-n+1}\frac{C(N-a_n,a_n)}{a_n}\cdot\frac{c}{(N-a_n)^2}\\
&\leq \frac{c}{n-1}\sum_{a_n=1}^{N-1}\frac{C(N-a_n,a_n)}{a_n(N-a_n)}\leq \frac{c}{n-1}\cdot\frac{8}{N^2}.
\end{align*}
Since $\frac{8}{n-1}\leq 1$ for $n\geq 9$, we can take a constant in the big-$O$ notation to be independent of $n$.
This proves the lemma.
\end{proof}
\begin{theorem}\label{thm:main_transport_relations}
Let $n$ and $d$ be integers satisfying $n\geq 2$ and $1\leq d\leq n$.
Let $\bk_1,\dots, \bk_d, \bl$ be indices and $\bk_{d+1},\dots, \bk_n$ non-empty indices.
Take tuples of complex numbers $\bz_i\in\DD^{\dep(\bk_i)}$ for all $1\leq i \leq n$, $\bw\in\DD^{\dep(\bl)}$ and complex numbers $v_1,\ldots,v_d\in\DD\setminus\{0\}$, $t\in \DD$ satisfying
\begin{equation}\label{eq:condition_for_variables}
\frac{1}{v_1}+\cdots+\frac{1}{v_d}=t
\end{equation}
and the assumption that if $n=d=2$ and $\bk_i=\emp$ $(1\leq i\leq 2)$, then $|v_{3-i}|<1$ or $|t|<1$.
Then we have
\begin{align*}
&\sum_{i=1}^{d}Z_n\left(\left({\bz_1 \atop \bk_1}\right)_{\overset{v_1}{\longrightarrow}};\dots;{\bz_i\atop \bk_i};\dots;\left({\bz_d \atop \bk_d}\right)_{\overset{v_d}{\longrightarrow}};\left({\bz_{d+1} \atop \bk_{d+1}}\right)_{\up};\dots; \left({\bz_{n} \atop \bk_{n}}\right)_{\up} \ \middle| \ \left({\bw \atop \bl}\right)_{\overset{t}{\longrightarrow}}\right)\\
&\quad -\sum_{i=d+1}^nZ_n\left(\left({\bz_1 \atop \bk_1}\right)_{\overset{v_1}{\longrightarrow}};\dots;\left({\bz_d \atop \bk_d}\right)_{\overset{v_d}{\longrightarrow}};\left({\bz_{d+1} \atop \bk_{d+1}}\right)_{\up};\dots;{\bz_i\atop \bk_i};\dots;\left({\bz_{n} \atop \bk_{n}}\right)_{\up} \ \middle| \ \left({\bw \atop \bl}\right)_{\overset{t}{\longrightarrow}}\right)\\
&=-Z_n\left(\left({\bz_1 \atop \bk_1}\right)_{\overset{v_1}{\longrightarrow}};\dots;\left({\bz_d \atop \bk_d}\right)_{\overset{v_d}{\longrightarrow}};\left({\bz_{d+1} \atop \bk_{d+1}}\right)_{\up};\dots;\left({\bz_{n} \atop \bk_{n}}\right)_{\up} \ \middle| \ \left({\bw \atop \bl}\right)\right).
\end{align*}
\end{theorem}
\begin{proof}
For vectors $\boldsymbol{x}=(x_1,\dots,x_d)$ and $\boldsymbol{y}=(y_1,\dots,y_d)$, we use the following notation in this proof:
\begin{align*}
\boldsymbol{x}\prec\boldsymbol{y} &\overset{\text{def}}{\Longleftrightarrow} x_j<y_j \ \text{for all} \ 1\leq j\leq d, \\ \boldsymbol{x}\prec_i\boldsymbol{y} &\overset{\text{def}}{\Longleftrightarrow} x_j<y_j \ \text{for all} \ 1\leq j\leq d, j\neq i \ \text{and} \ x_i=y_i,
\\\boldsymbol{x}\preceq_i\boldsymbol{y} &\overset{\text{def}}{\Longleftrightarrow} x_j<y_j \ \text{for all} \ 1\leq j\leq d, j\neq i \ \text{and} \ x_i\leq y_i.
\end{align*}
Then it is sufficient to show that an identity
\begin{align*}
&\frac{1}{m_{d+1}\cdots m_n}\sum_{i=1}^d\sum_{\substack{\ba\in\ZZ^d, \bm^-\prec_i\ba \\ q\leq|\ba|+|\bm^+|}}C(a_1,\dots,m_i,\dots,a_d,\bm^+)\cdot\left[\prod_{\substack{k=1 \\ k\neq i}}^d\frac{v_k^{a_k-m_k}}{a_k}\right]\cdot t^{|\ba|+|\bm^+|-q}\\
&\quad-\sum_{i=d+1}^n\frac{m_i}{m_{d+1}\cdots m_n}\sum_{\substack{\ba\in\ZZ^d, \bm^-\prec\ba \\ q\leq|\ba|+|\bm^+|}}C(a_1,\dots,a_d,\bm^+)\cdot\left[\prod_{k=1}^d\frac{v_k^{a_k-m_k}}{a_k}\right]\cdot t^{|\ba|+|\bm^+|-q}\\
&=-\frac{1}{m_{d+1}\cdots m_n}\sum_{\substack{\ba\in\ZZ^d, \bm^-\prec\ba \\ q=|\ba|+|\bm^+|}}C(a_1,\dots,a_d,\bm^+)\cdot\left[\prod_{k=1}^d\frac{v_k^{a_k-m_k}}{a_k}\right]\cdot q
\end{align*}
holds for non-negative integers $m_1,\ldots,m_d, q$ and positive integers $m_{d+1},\dots, m_n$.
Here, we abbreviated $(m_1,\dots,m_d)$ (resp.~$(m_{d+1},\dots,m_n)$) to $\bm^-$ (resp.~$\bm^+$).
We also abbreviated $(a_1,\dots, a_d)$ to $\ba$ for variables.
The symbol $|\ba|+|\bm^+|$ denotes $a_1+\cdots+a_d+m_{d+1}+\cdots+m_n$.
Note that $C(a_1,\dots,a_d,\bm^+)$ means $C(a_1,\dots,a_d,m_{d+1},\dots,m_n)$, not $C(a_1,\dots,a_d, (m_{d+1},\dots,m_n))$.
We see that each series above converges absolutely from the estimate \eqref{eq:connector ineq} or the assumption for the case where $n=d=2$.
Note that we cover the case where $t=0$ by considering $0^0$ as $1$.
For the $d=n$ case, we understand that $\bm^+=\emp$ and $m_{d+1}\cdots m_n=1$.
For the $d=1$ case, we understand the first term of the left-hand side of the above identity as 
\[
\frac{1}{m_{2}\cdots m_n}\cdot C(m_1,\dots, m_n)\cdot t^{m_1+\cdots +m_n-q}
\]
when $q\leq m_1+\cdots+m_n$ and $0$ otherwise.

Take a sufficiently large integer $N$.
For each $1\leq i\leq d$,
\begin{align*}
&\sum_{\substack{\ba\in\ZZ^d, \bm^-\prec_i\ba \\ q\leq|\ba|+|\bm^+|< N}}C(a_1,\dots,m_i,\dots,a_d,\bm^+)\cdot\left[\prod_{\substack{k=1 \\ k\neq i}}^d\frac{v_k^{a_k-m_k}}{a_k}\right]\cdot t^{|\ba|+|\bm^+|-q}\\
&=\sum_{\substack{\ba\in\ZZ^{d}, \bm^-\preceq_i\ba\\ q\leq|\ba|+|\bm^+|< N}}C(a_1,\dots,a_d,\bm^+)\cdot a_i\left[\prod_{k=1}^d\frac{v_k^{a_k-m_k}}{a_k}\right]\cdot t^{|\ba|+|\bm^+|-q}\\
&\quad\quad-\sum_{\substack{\ba\in\ZZ^{d}, \bm^-\prec\ba\\ q\leq|\ba|+|\bm^+|< N}}C(a_1,\dots,a_d,\bm^+)\cdot a_i\left[\prod_{k=1}^d\frac{v_k^{a_k-m_k}}{a_k}\right]\cdot t^{|\ba|+|\bm^+|-q}.
\end{align*}
For the first term of the right-hand side, by replacing $a_i$ with $a_i-1$,
\begin{align*}
&\sum_{\substack{\ba\in\ZZ^{d}, \bm^-\preceq_i\ba \\ q\leq|\ba|+|\bm^+|< N}}C(a_1,\dots,a_d,\bm^+)\cdot a_i\left[\prod_{k=1}^d\frac{v_k^{a_k-m_k}}{a_k}\right]\cdot t^{|\ba|+|\bm^+|-q}\\
&=\sum_{\substack{\ba\in\ZZ^{d}, \bm^-\prec\ba \\ q<|\ba|+|\bm^+|\leq N}}C(a_1,\dots,a_i-1,\dots,a_d,\bm^+)\cdot \frac{a_i}{v_i}\left[\prod_{k=1}^d\frac{v_k^{a_k-m_k}}{a_k}\right]\cdot t^{|\ba|+|\bm^+|-q-1}\\
&=\frac{1}{v_i}\sum_{\substack{\ba\in\ZZ^{d}, \bm^-\prec\ba \\ q<|\ba|+|\bm^+|\leq N}}(|\ba|+|\bm^+|)\cdot C(a_1,\dots,a_d,\bm^+)\cdot\left[\prod_{k=1}^d\frac{v_k^{a_k-m_k}}{a_k}\right]\cdot t^{|\ba|+|\bm^+|-q-1}.
\end{align*}
By summing from $i=1$ to $d$ and using \eqref{eq:condition_for_variables}, we have
\begin{align*}
&\sum_{i=1}^d\sum_{\substack{\ba\in\ZZ^{d}, \bm^-\preceq_i\ba \\ q\leq|\ba|+|\bm^+|< N}}C(a_1,\dots,a_d,\bm^+)\cdot a_i\left[\prod_{k=1}^d\frac{v_k^{a_k-m_k}}{a_k}\right]\cdot t^{|\ba|+|\bm^+|-q}\\
&=\sum_{\substack{\ba\in\ZZ^{d}, \bm^-\prec\ba \\ q<|\ba|+|\bm^+|\leq N}}(|\ba|+|\bm^+|)\cdot C(a_1,\dots,a_d,\bm^+)\cdot\left[\prod_{k=1}^d\frac{v_k^{a_k-m_k}}{a_k}\right]\cdot t^{|\ba|+|\bm^+|-q}.
\end{align*}
On the remaining terms,
\begin{align*}
&\frac{1}{m_{d+1}\cdots m_n}\sum_{i=1}^d\sum_{\substack{\ba\in\ZZ^{d}, \bm^-\prec\ba \\ q\leq|\ba|+|\bm^+|< N}}C(a_1,\dots,a_d,\bm^+)\cdot a_i\left[\prod_{k=1}^d\frac{v_k^{a_k-m_k}}{a_k}\right]\cdot t^{|\ba|+|\bm^+|-q}\\
&\quad+\sum_{i=d+1}^n\frac{m_i}{m_{d+1}\cdots m_n}\sum_{\substack{\ba\in\ZZ^{d}, \bm^-\prec\ba \\ q\leq|\ba|+|\bm^+|< N}}C(a_1,\dots,a_d,\bm^+)\cdot\left[\prod_{k=1}^d\frac{v_k^{a_k-m_k}}{a_k}\right]\cdot t^{|\ba|+|\bm^+|-q}\\
&=\frac{1}{m_{d+1}\cdots m_n}\sum_{\substack{\ba\in\ZZ^{d}, \bm^-\prec\ba \\ q\leq|\ba|+|\bm^+|< N}}(|\ba|+|\bm^+|)\cdot C(a_1,\dots,a_d,\bm^+)\cdot\left[\prod_{k=1}^d\frac{v_k^{a_k-m_k}}{a_k}\right]\cdot t^{|\ba|+|\bm^+|-q}.
\end{align*}
By combining these calculations, we have
\begin{align*}
&\frac{1}{m_{d+1}\cdots m_n}\sum_{i=1}^d\sum_{\substack{\ba\in\ZZ^{d}, \bm^-\prec_i\ba \\ q\leq|\ba|+|\bm^+|< N}}C(a_1,\dots,m_i,\dots,a_d,\bm^+)\cdot\left[\prod_{\substack{k=1 \\ k\neq i}}^d\frac{v_k^{a_k-m_k}}{a_k}\right]\cdot t^{|\ba|+|\bm^+|-q}\\
&\quad-\sum_{i=d+1}^n\frac{m_i}{m_{d+1}\cdots m_n}\sum_{\substack{\ba\in\ZZ^{d}, \bm^-\prec\ba \\ q\leq|\ba|+|\bm^+|< N}}C(a_1,\dots,a_d,\bm^+)\cdot\left[\prod_{k=1}^d\frac{v_k^{a_k-m_k}}{a_k}\right]\cdot t^{|\ba|+|\bm^+|-q}\\
&=\frac{1}{m_{d+1}\cdots m_n}\sum_{\substack{\ba\in\ZZ^{d}, \bm^-\prec\ba \\ q<|\ba|+|\bm^+|\leq N}}(|\ba|+|\bm^+|)\cdot C(a_1,\dots,a_d,\bm^+)\cdot\left[\prod_{k=1}^d\frac{v_k^{a_k-m_k}}{a_k}\right]\cdot t^{|\ba|+|\bm^+|-q}\\
&\quad\quad-\frac{1}{m_{d+1}\cdots m_n}\sum_{\substack{\ba\in\ZZ^{d}, \bm^-\prec\ba \\ q\leq|\ba|+|\bm^+|< N}}(|\ba|+|\bm^+|)\cdot C(a_1,\dots,a_d,\bm^+)\cdot\left[\prod_{k=1}^d\frac{v_k^{a_k-m_k}}{a_k}\right]\cdot t^{|\ba|+|\bm^+|-q}\\
&=-\frac{1}{m_{d+1}\cdots m_n}\sum_{\substack{\ba\in\ZZ^{d}, \bm^-\prec\ba \\ q=|\ba|+|\bm^+|}}C(a_1,\dots,a_d,\bm^+)\cdot\left[\prod_{k=1}^d\frac{v_k^{a_k-m_k}}{a_k}\right]\cdot q\\
&\quad\quad+\frac{N}{m_{d+1}\cdots m_n}\sum_{\substack{\ba\in\ZZ^{d}, \bm^-\prec\ba \\ |\ba|+|\bm^+|=N}}C(a_1,\dots,a_d,\bm^+)\cdot\left[\prod_{k=1}^d\frac{v_k^{a_k-m_k}}{a_k}\right]\cdot t^{N-q}.
\end{align*}
The last term tends to $0$ as $N$ tends to infinity by Lemma~\ref{lem:easy-bound}. This completes the proof.
\end{proof}
\section{Transport relations}\label{sec:transport relations}
Let $m$ be a positive integer.
For $t\in\DD$, we define $\BB_m(t)$ by
\[
\BB_m(t)\coloneqq\left\{(v_1,\dots,v_m)\in((\DD\setminus\{0\})\cup\{\infty\})^m \ \middle| \ \sum_{i=1}^m\frac{1}{v_i}=t \text{ or } \left|t-\sum_{i=1}^m\frac{1}{v_i}\right|\geq 1\right\}.
\]
For example, $\BB_1(1)=\{z\in\DD\setminus\{0\} \mid \mathrm{Re}(z)\leq \frac{1}{2}\}\cup\{1,\infty\}$, where $\mathrm{Re}(z)$ denotes the real part of $z$.

Let $n$ be an integer greater than $1$.
Let $\bk_1,\dots,\bk_n, \bl$ be indices and $\bz_1\in\DD^{\dep(\bk_1)},\dots,\bz_n\in\DD^{\dep(\bk_n)},\bw\in\DD^{\dep(\bl)}$.
Let $t\in \DD$ and $(v_1,\dots,v_{n-1})\in\BB_{n-1}(t)$.
We assume that
\begin{itemize}
\item if $\bk_i=\emp$ ($1\leq i \leq n-1$), then $v_i\neq\infty$,
\item if $\bk_n=\emp$, then $\sum_{i=1}^{n-1}\frac{1}{v_i}\neq t$,
\item if $n=2$, $\bk_1=\emp$ and $|t|=1$, then $\left|t-\frac{1}{v_1}\right|\neq 1$ and
\item if $n=2$, $\bk_2=\emp$ and $|t|=1$, then $|v_1|\neq 1$.
\end{itemize}
In this setting, by the fundamental identity, we have
\begin{equation}\label{eq:main transport relation}
\begin{split}
&Z_n\left(\left({\bz_1\atop\bk_1}\right)_{\overset{v_1}{\longrightarrow}};\dots;\left({\bz_{n-1}\atop\bk_{n-1}}\right)_{\overset{v_{n-1}}{\longrightarrow}};{\bz_n\atop\bk_n} \ \middle| \ \left({\bw\atop\bl}\right)_{\overset{t}{\longrightarrow}}\right)\\
&=-\sum_{i=1}^{n-1}Z_n\left(\left({\bz_1\atop\bk_1}\right)_{\overset{v_1}{\longrightarrow}};\dots;{\bz_i\atop\bk_i};\dots;\left({\bz_n\atop\bk_n}\right)_{\overset{v_n}{\longrightarrow}} \ \middle| \ \left({\bw\atop\bl}\right)_{\overset{t}{\longrightarrow}}\right)\\
&\quad-Z_n\left(\left({\bz_1\atop\bk_1}\right)_{\overset{v_1}{\longrightarrow}};\dots;\left({\bz_{n}\atop\bk_{n}}\right)_{\overset{v_{n}}{\longrightarrow}} \ \middle| \ {\bw\atop\bl}\right),
\end{split}
\end{equation}
where $v_n\in(\DD\setminus\{0\})\cup\{\infty\}$ is defined by $\frac{1}{v_n}=t-\sum_{i=1}^{n-1}\frac{1}{v_i}$.
In this equality, the total weight of indices (including the arrow parts) except for the $n$th index of each multivariable connected sum appearing in the right-hand side is $1$ less than the total weight of indices except for the $n$th index of the multivariable connected sum  in the left-hand side.
Equality~\eqref{eq:main transport relation} is the same as the fundamental identity, but when it is used for the purpose of transforming the left-hand side into the right-hand side, we also call it a \emph{transport relation}.
If we classify transport relations for $Z_n$ according to whether the actual direction of each of $\overset{v_1}{\rightarrow}$, $\dots,\overset{v_{n-1}}{\rightarrow}$ and $\overset{t}{\rightarrow}$ is vertical or horizontal, that is, whether each of $v_1,\dots,v_{n-1}$ (resp.~$t$) is the infinity (resp.~$0$) or not, there are $2^n$ different patterns of transport relations; see Example~\ref{ex:Z_2_gen} for instance.
\subsection{Examples}
The settings of $\bk_1,\bk_2,\bl,\bz_1,\bz_2$ and $\bw$ are as per those at the beginning of this section (the case of $n=2$).
\begin{example}[Transport relations for $Z_2$]\label{ex:Z_2_gen}
Let $t\in\DD$ and $v_1\in \BB_1(t)$.
Assume that if $\bk_2=\emp$, then $t\neq \frac{1}{v_1}$, and if $\bk_1=\emp$ (resp.~$\bk_2=\emp$) and $|t|=1$, then $\left|t-\frac{1}{v_1}\right|\neq 1$ (resp.~$|v_1|\neq 1$).
We define $v_2\in(\DD\setminus\{0\})\cup\{\infty\}$ by $\frac{1}{v_2}=t-\frac{1}{v_1}$.
\begin{enumerate}
\item\label{it:Ex3.1-1} The case of $v_1\neq\infty$, $t\neq0$:
\begin{align*}
&Z_2\left(\left({\bz_1\atop\bk_1}\right)_{\overset{v_1}{\longrightarrow}};{\bz_2\atop\bk_2} \ \middle| \ \left({\bw\atop\bl}\right)_{\overset{t}{\longrightarrow}}\right)\\
&=-Z_2\left({\bz_1\atop\bk_1};\left({\bz_2\atop\bk_2}\right)_{\overset{v_2}{\longrightarrow}} \ \middle| \ \left({\bw\atop\bl}\right)_{\overset{t}{\longrightarrow}}\right)-Z_2\left(\left({\bz_1\atop\bk_1}\right)_{\overset{v_1}{\longrightarrow}};\left({\bz_2\atop\bk_2}\right)_{\overset{v_2}{\longrightarrow}} \ \middle| \ {\bw\atop\bl}\right).
\end{align*}
\item\label{it:Ex3.1-2} The case of $v_1=\infty$, $t\neq 0$, $\bk_1\neq \emp$ ($v_2=1/t$):
\begin{align*}
&Z_2\left(\left({\bz_1\atop\bk_1}\right)_{\up};{\bz_2\atop\bk_2} \ \middle| \ \left({\bw\atop\bl}\right)_{\overset{t}{\longrightarrow}}\right)\\
&=Z_2\left({\bz_1\atop\bk_1};\left({\bz_2\atop\bk_2}\right)_{\overset{1/t}{\longrightarrow}} \ \middle| \ \left({\bw\atop\bl}\right)_{\overset{t}{\longrightarrow}}\right)-Z_2\left(\left({\bz_1\atop\bk_1}\right)_{\up};\left({\bz_2\atop\bk_2}\right)_{\overset{1/t}{\longrightarrow}} \ \middle| \ {\bw\atop\bl}\right).
\end{align*}
\item\label{it:Ex3.1-3} The case of $v_1\neq \infty$, $t=0$ ($v_2=-v_1$):
\begin{align*}
&Z_2\left(\left({\bz_1\atop\bk_1}\right)_{\overset{v_1}{\longrightarrow}};{\bz_2\atop\bk_2} \ \middle| \ \left({\bw\atop\bl}\right)_{\up}\right)\\
&=-Z_2\left({\bz_1\atop\bk_1};\left({\bz_2\atop\bk_2}\right)_{\overset{-v_1}{\longrightarrow}} \ \middle| \ \left({\bw\atop\bl}\right)_{\up}\right)-Z_2\left(\left({\bz_1\atop\bk_1}\right)_{\overset{v_1}{\longrightarrow}};\left({\bz_2\atop\bk_2}\right)_{\overset{-v_1}{\longrightarrow}} \ \middle| \ {\bw\atop\bl}\right).
\end{align*}
\item\label{it:Ex3.1-4} The case of $v_1=\infty$, $t=0$, $\bk_1\neq\emp$ ($v_2=\infty$):
\begin{align*}
&Z_2\left(\left({\bz_1\atop\bk_1}\right)_{\up};{\bz_2\atop\bk_2} \ \middle| \ \left({\bw\atop\bl}\right)_{\up}\right)\\
&=-Z_2\left({\bz_1\atop\bk_1};\left({\bz_2\atop\bk_2}\right)_{\up} \ \middle| \ \left({\bw\atop\bl}\right)_{\up}\right)+Z_2\left(\left({\bz_1\atop\bk_1}\right)_{\up};\left({\bz_2\atop\bk_2}\right)_{\up} \ \middle| \ {\bw\atop\bl}\right).
\end{align*}
\end{enumerate}
\end{example}
\begin{example}[Transport relations for duality]\label{ex:tranport relation for duality}
We consider a special case of Example~\ref{ex:Z_2_gen}: the case where $\bl=\emp$, $t=1$.
In the case of Example~\ref{ex:Z_2_gen} \eqref{it:Ex3.1-1}, we further classify it by whether the actual direction of $\overset{v_2}{\longrightarrow}$ is vertical or horizontal.
Let $v\in\DD\setminus\{0\}$ satisfying $\mathrm{Re}(v)\leq \frac{1}{2}$.
Assume that if $\bk_1=\emp$, then $\mathrm{Re}(v)\neq \frac{1}{2}$ and if $\bk_2=\emp$, then $|v|\neq 1$.
Then we have
\begin{align*}
&Z_2\left(\left({\bz_1\atop\bk_1}\right)_{\overset{v}{\longrightarrow}};{\bz_2\atop\bk_2}\right)=-Z_2\left({\bz_1\atop\bk_1};\left({\bz_2\atop\bk_2}\right)_{\overset{\frac{v}{v-1}}{\longrightarrow}}\right),\\
&Z_2\left(\left({\bz_1\atop\bk_1}\right)_{\overset{1}{\longrightarrow}};{\bz_2\atop\bk_2}\right)=Z_2\left({\bz_1\atop\bk_1};\left({\bz_2\atop\bk_2}\right)_{\up}\right)\quad (\bk_2\neq\emp),\\
&Z_2\left(\left({\bz_1\atop\bk_1}\right)_{\up};{\bz_2\atop\bk_2}\right)=Z_2\left({\bz_1\atop\bk_1};\left({\bz_2\atop\bk_2}\right)_{\overset{1}{\longrightarrow}}\right)\quad (\bk_1\neq\emp).
\end{align*}
These transport relations will be used to reprove the duality relation for MPLs in the next subsection, and recover those of Seki--Yamamoto \eqref{eq:SY-trans} by considering the case $\bz_1=(\{1\}^{\dep(\bk_1)})$ and $\bz_2=(\{1\}^{\dep(\bk_2)})$.
\end{example}
\begin{example}[Transport relations for special cases of $Z_3$]\label{ex:tranport relation for Z_3}
In the general case of $Z_3$, there are 8 patterns of transport relations, considered in the same way as in Example~\ref{ex:Z_2_gen}.
Here, we focus on transport relations restricted to the case where it is necessary to calculate the multiple connected sum $Z_3(\bk_1;\bk_2;\bk_3)$.
Let $\bk_3$ be an index and $\ee$ an element of $\{1,-1\}^{\dep(\bk_3)}$.
If the tuple of complex variables to be written onto an index has the form $(1,\dots,1)$, then  we omit the tuple from our notation.
In this setting, we have
\begin{enumerate}
\item\label{it:Ex3.3-1} The case of $(v_1,v_2,v_3)=(1,1,-1)$:
\[
Z_3\left((\bk_1)_{\to};(\bk_2)_{\to};{\ee\atop\bk_3}\right)=-Z_3\left(\bk_1;(\bk_2)_{\to};\left({\ee\atop\bk_3}\right)_{\overset{-1}{\longrightarrow}}\right)-Z_3\left((\bk_1)_{\to};\bk_2;\left({\ee\atop\bk_3}\right)_{\overset{-1}{\longrightarrow}}\right).
\]
\item\label{it:Ex3.3-2} The case of $(v_1,v_2,v_3)=(1,\infty,\infty)$, $\bk_2,\bk_3\neq\emp$:
\[
Z_3\left((\bk_1)_{\to};(\bk_2)_{\up};{\ee\atop\bk_3}\right)=Z_3\left(\bk_1;(\bk_2)_{\up};\left({\ee\atop\bk_3}\right)_{\up}\right)-Z_3\left((\bk_1)_{\to};\bk_2;\left({\ee\atop\bk_3}\right)_{\up}\right).
\]
\item\label{it:Ex3.3-3} The case of $(v_1,v_2,v_3)=(\infty,1,\infty)$, $\bk_1,\bk_3\neq\emp$:
\[
Z_3\left((\bk_1)_{\up};(\bk_2)_{\to};{\ee\atop\bk_3}\right)=-Z_3\left(\bk_1;(\bk_2)_{\to};\left({\ee\atop\bk_3}\right)_{\up}\right)+Z_3\left((\bk_1)_{\up};\bk_2;\left({\ee\atop\bk_3}\right)_{\up}\right).
\]
\item\label{it:Ex3.3-4} The case of $(v_1,v_2,v_3)=(\infty,\infty,1)$, $\bk_1,\bk_2\neq\emp$:
\[
Z_3\left((\bk_1)_{\up};(\bk_2)_{\up};{\ee\atop\bk_3}\right)=Z_3\left(\bk_1;(\bk_2)_{\up};\left({\ee\atop\bk_3}\right)_{\overset{1}{\longrightarrow}}\right)+Z_3\left((\bk_1)_{\up};\bk_2;\left({\ee\atop\bk_3}\right)_{\overset{1}{\longrightarrow}}\right).
\]
\end{enumerate}
In particular, if all $\bk_1, \bk_2$ and $\bk_3$ are not empty, then we see that $Z_3(\bk_1;\bk_2;\bk_3)$ belongs to the space of alternating multiple zeta values, which will be defined below, by the recipe explained later in \S\ref{subsec:recipe for evaluations}.
\end{example}
\subsection{Duality for multiple polylogarithms}\label{subsec:duality for MPL}
Prior to \S\ref{sec:Relations}, we prove the duality for MPLs by using transport relations in Example~\ref{ex:tranport relation for duality}, because of its simplicity.

Let $\bk$ be an index and $\bz=(z_1,\dots,z_r)\in\DD^r$, where $r\coloneqq\dep(\bk)$.
Further, we assume that  all of the following are true:
\begin{itemize}
\item $z_1,\dots,z_r\in \BB_1(1)$,
\item $\mathrm{Re}(z_1)\neq \frac{1}{2}$,
\item if $\bk$ is non-admissible, then $|z_r|\neq 1$.
\end{itemize}
We refer this condition as the \emph{dual condition}.
For such a pair $\left({\bz\atop\bk}\right)$, by using transport relations in Example~\ref{ex:tranport relation for duality}, we can transport $\left({\bz\atop\bk}\right)$ from left to right in the absolutely convergent multivariable connected sum $Z_2$ as
\begin{equation}\label{eq:Z_2 duality and invertible index}
Z_2\left({\bz\atop\bk};{\emp\atop\emp}\right)=\cdots=(-1)^{\iota(\bz)}Z_2\left({\emp\atop\emp};\left({\bz\atop\bk}\right)^{\dagger}\right),
\end{equation}
where $\iota(\bz)=r-$(the number of $1$'s in $\bz$).
Here, we define a pair $\left({\bz\atop\bk}\right)^{\dagger}$ satisfying the dual condition by this transportation.
We also write an explicit definition.
Recall the definition of the dual index of an admissible index in \S\ref{sec:Intro}.
Any pair $\left({\bz\atop\bk}\right)$ with the dual condition is uniquely expressed as
\[
\left({\bz\atop\bk}\right)=\left({\{1\}^{r_1},\atop\bl_1,}{\{1\}^{a_1-1},\atop\{1\}^{a_1-1},}{w_1\atop b_1}{,\dots,\atop,\dots,}{\{1\}^{r_d},\atop\bl_d,}{\{1\}^{a_d-1},\atop\{1\}^{a_d-1},}{w_d,\atop b_d,}{\{1\}^{r_{d+1}}\atop\bl_{d+1}}\right),
\]
where $d$ is a non-negative integer, $a_1,\dots, a_d$, $b_1,\dots, b_d$ are positive integers, all $w_1,\dots, w_d$ are not $1$, $\bl_1,\dots,\bl_{d+1}$ are admissible indices, and $r_1\coloneqq\dep(\bl_1),\dots, r_{d+1}\coloneqq\dep(\bl_{d+1})$.
Then $\left({\bz\atop\bk}\right)^{\dagger}$ is defined by
\[
\left({\bz\atop\bk}\right)^{\dagger}=\left({\{1\}^{s_{d+1}},\atop(\bl_{d+1})^{\dagger},}{\{1\}^{b_d-1},\atop\{1\}^{b_d-1},}{\frac{w_d}{w_d-1},\atop a_d,}{\{1\}^{s_d}\atop(\bl_d)^{\dagger}}{,\dots,\atop,\dots,}{\{1\}^{b_1-1},\atop\{1\}^{b_1-1},}{\frac{w_1}{w_1-1},\atop a_1,}{\{1\}^{s_{1}}\atop(\bl_{1})^{\dagger}}\right),
\]
where $s_1\coloneqq\dep((\bl_1)^{\dagger}),\dots, s_{d+1}\coloneqq\dep((\bl_{d+1})^{\dagger})$.
In particular, we have $\left({\{1\}^r \atop\bk}\right)^{\dagger}=\left({\{1\}^{s}\atop\bk^{\dagger}}\right)$, where $\bk$ is an admissible index and $r\coloneqq\dep(\bk)$, $s\coloneqq\dep(\bk^{\dagger})$.
The transportation \eqref{eq:Z_2 duality and invertible index} with Proposition~\ref{prop:fundamental_properties} \eqref{it:2.1-1} and \eqref{item:reduction} gives the following duality relation which was proved by Borwein--Bradley--Broadhurst--Lison\v{e}k in \cite[\S6.1]{Borwein-Bradley-Broadhurst-Lisonek} using a change of variables for iterated integrals; here, we state only the case for absolutely convergent MPLs.
\begin{theorem}\label{thm:duality}
Let $\bk$ be an index and $\bz\in\DD^{\dep(\bk)}$.
We assume that the pair $\left({\bz\atop\bk}\right)$ satisfies the dual condition and write $\left({\bz\atop\bk}\right)^{\dagger}$ as $\bigl({\bz'\atop\bk'}\bigr)$.
Then we have
\[
\Li_{\bk}(\bz)=(-1)^{\iota(\bz)}\Li_{\bk'}(\bz').
\]
\end{theorem}
Here, we recall the definition of the \emph{alternating multiple zeta value}.
For a non-empty index $\bk=(k_1,\dots,k_r)$ and $\ee=(\epsilon_1,\dots,\epsilon_r)\in\{1,-1\}^r$ satisfying $(k_r,\epsilon_r)\neq (1,1)$, we define $\zeta(\bk;\ee)$ by
\[
\zeta(\bk;\ee)\coloneqq\sum_{0<m_1<\cdots<m_r}\frac{\epsilon_1^{m_1}\cdots \epsilon_r^{m_r}}{m_1^{k_1}\cdots m_r^{k_r}}.
\]
We use bar notation for indicating signs; e.g.~$\zeta(\overline{1},2,\overline{3},4)=\zeta(1,2,3,4;-1,1,-1,1)$. 
\begin{example}
The identity
\[
\zeta(\overline{1},2)=-\Li_{2,1}\left(1,\frac{1}{2}\right)
\]
is deduced from
\[
Z_2\left({-1,1\atop \phantom{-}1,2};{\emp\atop\emp}\right)=Z_2\left({-1,1\atop \phantom{-}1,1};{1\atop1}\right)=Z_2\left({-1\atop \phantom{-}1};{1\atop 2}\right)=-Z_2\left({\emp\atop\emp};{1,\frac{1}{2}\atop 2,1}\right).
\]
Borwein, Bradley, Broadhurst and Lison\v{e}k remarked that this is ``a result that would doubtless be difficult to prove by na\"ive series manipulations alone'' in \cite[Example~6.1]{Borwein-Bradley-Broadhurst-Lisonek}.
However, our proof is contrary to their prediction; if we state the proof of Theorem~\ref{thm:main_transport_relations} in this case, then it is quite simple.
\end{example}
\begin{example}
Let $r$ be a positive integer.
Take $(z_1,\dots,z_r)\in\left(\BB_1(1)\setminus\{\infty\}\right)^r$ satisfying $\mathrm{Re}(z_1)\neq\frac{1}{2}$ and $|z_r|\neq 1$.
Then the identity
\[
\Li_{\{1\}^r}(z_1,\dots,z_r)=(-1)^r\Li_{\{1\}^r}\left(\frac{z_r}{z_r-1},\dots,\frac{z_1}{z_1-1}\right)
\]
is deduced from
\[
Z_2\left({z_1,\dots,z_r\atop 1,\dots, 1};{\emp\atop\emp}\right)=-Z_2\left({z_1\atop 1} {,\dots,\atop,\dots,} {z_{r-1}\atop 1};{\frac{z_r}{z_r-1}\atop1}\right)=\cdots=(-1)^rZ_2\left({\emp\atop\emp};{\frac{z_r}{z_r-1},\dots,\frac{z_1}{z_1-1}\atop 1,\dots, 1}\right).
\]
\end{example}
\section{Boundary condition}\label{sec:boundary condtions}
We call the following fact the \emph{boundary condition} for multivariable connected sums: every absolutely convergent $Z_1$-value is written explicitly as a $\ZZ$-linear combination of absolutely convergent shuffle-type MPLs.
We explain this fact in detail.

Let $r$ be a positive integer, $\bk=(k_1,\dots,k_r)$ an index and $\xi_1,\dots, \xi_r$ complex numbers satisfying $\left|\prod_{i=j}^r\xi_i\right|\leq 1$ for all $1\leq j\leq r$ and that if $k_r=1$, then $|\xi_r|<1$.
Then the \emph{harmonic-type multiple polylogarithm} $\mathrm{Li}^*_{\bk}(\xi_1,\dots,\xi_r)$ is defined by
\[
\mathrm{Li}^*_{\bk}(\xi_1,\dots,\xi_r)\coloneqq\sum_{0<m_1<\cdots<m_r}\frac{\xi_1^{m_1}\cdots \xi_r^{m_r}}{m_1^{k_1}\cdots m_r^{k_r}}.
\]
Let $\bz=(z_1,\dots,z_r)\in (\DD\setminus\{0\})^r$ satisfying that if $\bk$ is non-admissible, then $|z_r|<1$. 
By definition, there are the following relationships between $\mathrm{Li}^*_{\bk}$ and $\Li_{\bk}$:
\begin{align}
\Li_{\bk}(z_1,\dots,z_r)&=\mathrm{Li}^*_{\bk}\left(\frac{z_1}{z_2},\dots,\frac{z_{r-1}}{z_r},z_r\right),\\
\mathrm{Li}_{\bk}^*(\xi_1,\dots,\xi_r)&=\Li_{\bk}\left(\prod_{i=1}^r\xi_i,\prod_{i=2}^r\xi_i, \dots,\xi_{r-1}\xi_r, \xi_r\right).\label{eq:stuffle <-> shuffle}
\end{align}
Furthermore, let $s$ be a positive integer, $\bl=(l_1,\dots,l_s)$ an index and $\bw=(w_1,\dots w_s)\in (\DD\setminus\{0\})^s$.
We replace the assumption about $z_r$ with the following: if $k_r=l_s=1$, then $|z_r|<1$ or $|w_s|<1$.
By definition, we have
\begin{multline*}
Z_1\left({\bz\atop\bk} \ \middle| \ {\bw\atop\bl}\right)=\sum_{n=1}^{\infty}\left(\sum_{0<m_1<\cdots<m_{r-1}<n}\frac{\left(\frac{z_1}{z_2}\right)^{m_1}\cdots\left(\frac{z_{r-1}}{z_r}\right)^{m_{r-1}}}{m_1^{k_1}\cdots m_{r-1}^{k_{r-1}}}\right)\\
\times \left(\sum_{1\leq q_1\leq \cdots \leq q_{s-1}\leq n}\frac{\left(\frac{w_1}{w_2}\right)^{q_1}\cdots \left(\frac{w_{s-1}}{w_s}\right)^{q_{s-1}}}{q_1^{l_1}\cdots q_{s-1}^{l_{s-1}}}\right)\frac{(z_rw_s)^n}{n^{k_r+l_s-1}}.
\end{multline*}
Since the product of truncated sums for a fixed $n$ obeys the harmonic product rule, $Z_1\left({\bz\atop\bk} \ \middle| \ {\bw\atop\bl}\right)$ is written as a $\ZZ$-linear combination of harmonic-type MPLs, and hence of shuffle-type by \eqref{eq:stuffle <-> shuffle}.
For example, we have
\begin{align*}
&Z_1\left({z_1,z_2\atop 1,1} \ \middle| \ {w_1,w_2\atop 1,2}\right)\\
&=\sum_{n=1}^{\infty}\left(\sum_{0<m_1<n}\frac{\left(\frac{z_1}{z_2}\right)^{m_1}}{m_1}\right)\left(\sum_{0<q_1<n}\frac{\left(\frac{w_1}{w_2}\right)^{q_1}}{q_1}+\frac{\left(\frac{w_1}{w_2}\right)^n}{n}\right)\frac{(z_2w_2)^n}{n^2}
\\
&=\mathrm{Li}_{1,1,2}^*\left(\frac{z_1}{z_2},\frac{w_1}{w_2},z_2w_2\right)+\mathrm{Li}_{1,1,2}^*\left(\frac{w_1}{w_2},\frac{z_1}{z_2},z_2w_2\right)+\mathrm{Li}_{2,2}^*\left(\frac{z_1w_1}{z_2w_2},z_2w_2\right)+\mathrm{Li}_{1,3}^*\left(\frac{z_1}{z_2},z_2w_1\right)\\
&=\Li_{1,1,2}(z_1w_1,z_2w_1,z_2w_2)+\Li_{1,1,2}(z_1w_1,z_1w_2,z_2w_2)+\Li_{2,2}(z_1w_1,z_2w_2)+\Li_{1,3}(z_1w_1,z_2w_1).
\end{align*}
\begin{remark}
If $z_1=\cdots=z_r=w_1=\cdots=w_s=1$, then another decomposition approach using $2$-posets and associated integrals is attributable to Kaneko and Yamamoto; see \cite[\S4]{Kaneko-Yamamoto}.
\end{remark}
\section{Evaluations of the multivariable connected sum}\label{sec:Evaluations}
\subsection{Recipe for evaluations}\label{subsec:recipe for evaluations}
First, we introduce the notion of transportable variables.
This is just a condition to complete all transportation process while maintaining absolute convergence of multivariable connected sums.
We use the notation $[n]\coloneqq \{1,\dots,n\}$ for a positive integer $n$.
\begin{definition}\label{def:transportable}
Let $r_1,\dots, r_n$ and $s$ be positive integers, where $n$ is an integer at least $2$.
Let $\bz_1=(z_1^{(1)},\dots,z_{r_1}^{(1)}),\dots,\bz_n=(z_1^{(n)},\dots,z_{r_n}^{(n)})$ and $\bw=(w_1,\dots, w_s)$ be elements of $\DD^{r_1}, \dots, \DD^{r_n}$ and $\DD^{s}$, respectively.
We declare that $(\bz_1,\dots,\bz_n;\bw)$ is \emph{transportable} if there exists $j\in[n]$ such that for any non-empty subset $J\subset[n]\setminus\{j\}$, the following holds true:
\begin{itemize}
\item For every $(a_i)_{i\in J}$ ($a_i\in[r_i]$) and every component $w_k$ of $\bw$,
\[
(z_{a_i}^{(i)})_{i\in J}\in\BB_{\#J}(w_k).
\]
\item For every component $w_k$ of $\bw$ with $|w_k|=1$,
\[
\left|w_k-\frac{1}{z_1^{(i)}}\right|\neq 1,\quad \text{for all } i\in [n]\setminus\{j\}.
\]
\end{itemize}
\end{definition}
If every $\bz_i$ is an element of $\{1,-1\}^{r_i}$ and $\bw\in\{0,1,-1\}^s$, then $(\bz_1,\dots,\bz_n;\bw)$ is transportable by definition.

Given the above, we can provide an explicit decomposition of the multivariable connected sum according to the following recipe:
\begin{theorem}\label{thm:main1 precise}
Let $r_1,\dots, r_n$ and $s$ be positive integers, where $n$ is an integer at least $2$.
Let $\bk_1,\dots,\bk_n$ and $\bl$ be indices satisfying $\dep(\bk_1)=r_1,\dots,\dep(\bk_n)=r_n$, $\dep(\bl)=s$.
Let $\bz_1,\dots,\bz_n$ and $\bw$ be elements of $\DD^{r_1}, \dots, \DD^{r_n}$ and $\DD^{s}$, respectively.
We assume that if $\bl=\{1\}^s$ $($resp.~$\bl\neq\{1\}^s$$)$, then $(\bz_1,\dots,\bz_n;\bw)$ $($resp.~$(\bz_1,\dots,\bz_n;(\bw,0))$$)$ is transportable.
Then the multivariable connected sum $Z_n\left({\bz_1\atop\bk_1};\dots;{\bz_n\atop\bk_n} \ \middle| \ {\bw\atop\bl}\right)$
is expressed as a $\ZZ$-linear combination of absolutely convergent multiple polylogarithms.
Explicitly, we obtain such an expression by the following procedures:
\begin{enumerate}[\rm(i)]
\item\label{item:recipe ev-1} Take an integer $j\in[n]$ in Definition~$\ref{def:transportable}$ and replace the $j$th  component with the $n$th component by Proposition~$\ref{prop:fundamental_properties}$ \eqref{item:symmetry}.
\item\label{item:recipe ev-2} Apply the transport relation \eqref{eq:main transport relation} repeatedly until $Z_n\left({\bz_1\atop\bk_1};\dots;{\bz_n\atop\bk_n} \ \middle| \ {\bw\atop\bl}\right)$ is expressed as a $\ZZ$-linear combination of $Z_{n-1}$-values by Proposition~$\ref{prop:fundamental_properties}$ \eqref{item:reduction}.
\item\label{item:recipe ev-3} If $n\geq 3$, then apply the transport relation \eqref{eq:main transport relation} $($replace $n$ with $n-1$$)$ to each $Z_{n-1}$-value repeatedly until it is expressed as a $\ZZ$-linear combination of $Z_{n-2}$-values.
\item\label{item:recipe ev-4} Repeat the same procedure as in \eqref{item:recipe ev-3} to express $Z_n\left({\bz_1\atop\bk_1};\dots;{\bz_n\atop\bk_n} \ \middle| \ {\bw\atop\bl}\right)$ as a $\ZZ$-linear combination of $Z_{1}$-values.
\item\label{item:recipe ev-5} Apply the boundary condition in \S$\ref{sec:boundary condtions}$ to each $Z_1$-value obtained in \eqref{item:recipe ev-4}; then, we have the desired $\ZZ$-linear combination of absolutely convergent multiple polylogarithms.
\end{enumerate}
\end{theorem}
\begin{proof}
All that needs to be checked is that sufficient conditions for applying the transport relation in \eqref{item:recipe ev-2}--\eqref{item:recipe ev-4} are satisfied, which is accomplished by the assumption and Definition~\ref{def:transportable}.
\end{proof}
\begin{proof}[{Proof of Theorem~$\ref{thm:main_with_no_variables}$}]
Let $k$ be a positive integer and $\bk_1,\dots,\bk_n,\bl$ indices as in the statement.
We decompose $Z_n(\bk_1;\dots;\bk_n\mid \bl)$ as a $\ZZ$-linear combination of $Z_1$-values according to the recipe \eqref{item:recipe ev-2}--\eqref{item:recipe ev-4} in Theorem~\ref{thm:main1 precise}; the transportable condition is satisfied for all $j\in [n]$ in this case and thus we may assume that $j=n$.
Then we see that each the $Z_1$-value belongs to
\[
\left\{Z_1\left({1,z_2,\dots, z_r \atop \bk} \ \middle| \ {\{1\}^s\atop \bk'}\right) \ \middle| \  
\begin{array}{c}
r,s\in\ZZ_{>0}, r+s\leq k-1,\\
\bk\in I_r,\bk'\in I_s \ \text{with } \wt(\bk)+\wt(\bk')=k, \\
\bk \ \text{or } \bk' \ \text{is admissible}, \text{and } \\
z_2,\dots,z_r\in \{1,-1,-\frac{1}{2}\dots,-\frac{1}{n-1}\}
\end{array}
\right\},
\]
where $I_r$ (resp.~$I_s$) denotes the set of indices whose depth is $r$ (resp.~$s$).
By the boundary condition in \S$\ref{sec:boundary condtions}$, it is straightforward to see that every $Z_1$-value belonging to the above set is expressed as a $\ZZ$-linear combination of elements of
\[
\left\{\Li_{\bk}(1,z_2,\dots, z_r) \ \middle| \ {1\leq r\leq k-2, \bk \ \text{is admissible with } \wt(\bk)=k-1, \atop \dep(\bk)=r, \text{and } z_2,\dots,z_r\in \{1,-1,-\frac{1}{2}\dots,-\frac{1}{n-1}\}}\right\}.
\]
Finally, every element of this set coincides with an element of $\mathrm{MPL}(n,k)$ up to sign by the duality (Theorem~\ref{thm:duality}).
\end{proof}
\subsection{Examples}
\begin{example}[= Equality~\eqref{eq:Oloa}]
\[
Z_2(1;1\mid 1,1)=3\zeta(3).
\]
\end{example}
\begin{proof}
By the transport relation, we have
\[
Z_2(1;1\mid 1,1)=Z_2(\emp;2\mid 1,1)+Z_2(1;2)=Z_1(2\mid 1,1)+Z_1(3).
\]
Boundary conditions for this case are $Z_1(2\mid 1,1)=\zeta(1,2)+\zeta(3)=2\zeta(3)$ and $Z_1(3)=\zeta(3)$.
\end{proof}
\begin{example}[= Equality~\eqref{eq:Oloa-like}]\label{ex:new formula for zeta(4)}
\[
Z_2(1;1\mid 2,1)=\frac{17}{8}\zeta(4).
\]
\end{example}
\begin{proof}
By the transport relation, we have
\begin{align*}
Z_2(1;1\mid 2,1)&=Z_2(\emp;2\mid 2,1)+Z_2(1;2\mid 2)\\
&=Z_1(2\mid 2,1)-Z_2\left({\emp\atop\emp};{1,-1\atop2,\phantom{-}1}\ \middle| \  {1\atop 2}\right)-Z_2\left({1\atop1};{1,-1\atop2,\phantom{-}1}\right)\\
&=Z_1(2\mid 2,1)-Z_1\left({1,-1\atop2,\phantom{-}1}\ \middle| \  {1\atop 2}\right)-Z_1\left({1,-1\atop 2,\phantom{-}2}\right).
\end{align*}
Boundary conditions for this case are
\[
Z_1(2\mid 2,1)=\zeta(2,2)+\zeta(4),\quad Z_1\left({1,-1\atop2,\phantom{-}1}\ \middle| \ {1\atop 2}\right)=\zeta(\overline{2},\overline{2}),\quad Z_1\left({1,-1\atop 2,\phantom{-}2}\right)=\zeta(\overline{2},\overline{2}).
\]
Hence we have
\[
Z_2(1;1\mid 2,1)=\zeta(2,2)+\zeta(4)-2\zeta(\overline{2},\overline{2}).
\]
The facts that $\zeta(2,2)=\frac{3}{4}\zeta(4)$ and $\zeta(\overline{2},\overline{2})=-\frac{3}{16}\zeta(4)$ are well-known.
\end{proof}
\begin{example}[= Equality~\eqref{eq:Amtagpa}]
For a positive integer $n$, we have
\[
Z_{n+1}(1;1;\dots;1)=n!\Li_{\{1\}^{n-1},2}\left(1,\frac{1}{2},\dots,\frac{1}{n}\right).
\]
\end{example}
\begin{proof}
By the transport relation \eqref{eq:main transport relation} and Proposition~$\ref{prop:fundamental_properties}$ \eqref{item:reduction},
\[
Z_{r+1}\left({1\atop 1};\dots;{1\atop 1};{\bz\atop\bk}\right)=-rZ_{r}\left({1\atop 1};\dots;{1\atop 1};\left({\bz\atop\bk}\right)_{\overset{\frac{1}{1-r}}{\longrightarrow}}\right)
\]
holds for an integer $r\geq 2$, a non-empty index $\bk$ and $\bz\in\DD^{\dep(\bk)}$.
Applying this repeatedly and then applying the transport relation for the duality, we have
\begin{align*}
Z_{n+1}(1;1;\dots;1)&=(-1)^{n-1}n!Z_2\left({1\atop 1};{1, \atop 1,}{-\frac 1{n-1} \atop \phantom{-}1}{,\dots,\atop,\dots,}{-\frac{1}{2},\atop\phantom{-}1,}{-1\atop \phantom{-}1}\right)\\
&=(-1)^{n-1}n!\Li_{\{1\}^{n-1},2}\left(1,-\frac 1{n-1},\dots,-\frac 12,-1\right).
\end{align*}
Since 
\begin{equation}\label{eq:11112duality}
\Li_{\{1\}^{n-1},2}\left(1,-\frac 1{n-1},\dots,-\frac 12,-1\right)=(-1)^{n-1}\Li_{\{1\}^{n-1},2}\left(1,\frac{1}{2},\dots,\frac{1}{n}\right)
\end{equation}
holds as a consequence of the duality (Theorem~\ref{thm:duality}), the proof is completed.
\end{proof}
\begin{example}[Equality~\eqref{eq:Cloitre} with variables]\label{ex:var Cloitre}
Let $z_1$ and $z_2$ be complex numbers whose absolute values are less than $1$.
Assume that $\mathrm{Re}(z_i)<\frac{1}{2}$ holds for $i=1$ or $2$.
Then we have
\[
Z_2\left({z_1\atop 1};{z_2\atop 1}\right)=\mathrm{Li}_2(z_1)+\mathrm{Li}_2(z_2)-\mathrm{Li}_2(z_1+z_2-z_1z_2).
\]
Here, $\mathrm{Li}_2(z)$ is the dilogarithm.
\end{example}
\begin{proof}
We may assume that $z_1\neq 0$ and $\mathrm{Re}(z_1)<\frac{1}{2}$.
By the transport relation and Proposition~$\ref{prop:fundamental_properties}$ \eqref{item:reduction}, \eqref{it:2.1-1}, we have
\begin{equation}\label{eq:Z11}
Z_2\left({z_1\atop 1};{z_2\atop 1}\right)=-Z_2\left({\emp\atop\emp};{z_2, \atop 1,} {\frac{z_1}{z_1-1}\atop 1}\right)=-\Li_{1,1}\left(z_2,\frac{z_1}{z_1-1}\right).
\end{equation}
By substituting $z_2(z_1-1)/z_1$ (resp.~$z_1/(z_1-1)$) into $x$ (resp.~$y$) in Zagier's formula \cite[Proposition~1]{Zagier}
\[
\mathrm{Li}^*_{1,1}(x,y)=\mathrm{Li}_2\left(\frac{xy-y}{1-y}\right)-\mathrm{Li}_2\left(\frac{y}{y-1}\right)-\mathrm{Li}_2(xy) \quad (|xy|<1, |y|<1),
\]
we have the desired formula.
\end{proof}
\section{Relations among special values of multiple polylogarithms}\label{sec:Relations}
\subsection{Recipe for relations}\label{subsec:Recipe for relations}
We have already derived the duality for MPLs by using our transport relations in \S\ref{subsec:duality for MPL}.
In this subsection, we describe a general recipe for obtaining relations among MPLs based on transport relations and evaluations of multivariable connected sums (Theorem~\ref{thm:main1 precise}).

We consider data
\begin{itemize}
\item $n$: an integer at least $2$,
\item $\bk_1,\dots,\bk_{n-1}$ and $\bl$: non-empty indices,
\item $\bz_1=(z_1^{(1)},\dots, z_{r_1}^{(1)})\in\DD^{r_1},\dots, \bz_{n-1}=(z_1^{(n-1)},\dots,z_{r_{n-1}}^{(n-1)})\in\DD^{r_{n-1}}$ and $\bw=(w_1,\dots,w_s)\in\DD^s$, where $r_1=\dep(\bk_1),\dots,r_{n-1}=\dep(\bk_{n-1})$ and $s=\dep(\bl)$
\end{itemize}
satisfying the following assumptions:
\begin{itemize}
\item If $\bl$ is non-admissible, then
\[
\sum_{\substack{1\leq i\leq n-1 \\ \bk_i \ \text{is non-admissible}}}\frac{1}{z_{r_i}^{(i)}}\neq w_s,
\]
and if $\bl$ is admissible, then
\[
\sum_{\substack{1\leq i\leq n-1 \\ \bk_i \ \text{is non-admissible}}}\frac{1}{z_{r_i}^{(i)}}\neq 0.
\]
\item For every non-empty subset $J\subset[n-1]$, every $(a_i)_{i\in J}$ ($a_i\in[r_i]$) and every component $w_k$ of $\bw$,
\[
(z_{a_i}^{(i)})_{i\in J}\in\BB_{\#J}(w_k).
\]
\item If $\bl\neq \{1\}^s$, then the second assumption with $\bw$ replaced by $(\bw,0)$ is valid.
\item For every component $w_k$ of $\bw$ with $|w_k|=1$,
\[
\left|w_k-\frac{1}{z_1^{(i)}}\right|\neq 1,\quad \text{for all } i\in [n-1].
\]
\item If $n=2$ and both $\bk_1$ and $\bl$ are non-admissible, then $|z_{r_1}^{(1)}|<1$ or $|w_s|<1$ holds.
\end{itemize}
For such given data, we can obtain a relation among MPLs by the following procedure:
\begin{enumerate}[(i)]
\item\label{item:rel recipe-1} By Proposition~\ref{prop:fundamental_properties} \eqref{item:reduction}, we have
\begin{equation}\label{eq:first step of relations recipe}
Z_n\left({\bz_1\atop\bk_1};\dots;{\bz_{n-1}\atop\bk_{n-1}};{\emp\atop\emp} \ \middle| \ {\bw\atop\bl}\right)=Z_{n-1}\left({\bz_1\atop\bk_1};\dots;{\bz_{n-1}\atop\bk_{n-1}} \ \middle| \ {\bw\atop\bl}\right)
\end{equation}
and then we can decompose the right-hand side as a $\ZZ$-linear combination of MPLs by Theorem~\ref{thm:main1 precise} or the boundary condition.
\item\label{item:rel recipe-2} On the other hand, by using the transport relation \eqref{eq:main transport relation} and Propositions~\ref{prop:fundamental_properties} \eqref{item:reduction} as needed, we can decompose the left-hand side of \eqref{eq:first step of relations recipe} as a $\ZZ$-linear combination of $Z_n$-values and $Z_{n-1}$-values.
Then, we decompose all these $Z_n$-values and $Z_{n-1}$-values as $\ZZ$-linear combinations of MPLs by Theorem~\ref{thm:main1 precise} or the boundary condition.
\item\label{item:rel recipe-3} By comparing the output results of \eqref{item:rel recipe-1} and \eqref{item:rel recipe-2}, we obtain a relation among MPLs.
\end{enumerate}

Using assumptions regarding the data and contents in \S\ref{sec:transport relations}--\S\ref{sec:Evaluations}, it is true that the above procedure yields a relation among MPLs; note that the relation can be tautological as in \eqref{eq:Tautological}.
The family of relations obtained by executing the above procedure on data satisfying $n=2$, $\bl=(\{1\}^s)$ and $\bw=(\{1\}^s)$ is actually a natural extension of the family of Ohno's relations for MZVs to the MPLs case.
We prove Ohno's relation for MPLs in the next subsection; see Theorem~\ref{thm:Ohno MPLs}.

The above description is the concrete content of Theorem~\ref{thm:main2}.
\subsection{Ohno's relation for multiple polylogarithms}
For a non-negative integer $h$, we use the following shorthand notation for multivariable connected sums:
\[
Z^{(h)}\left({\bz_1\atop\bk_1};{\bz_2\atop\bk_2}\right)\coloneqq Z_2\left({\bz_1\atop\bk_1};{\bz_2\atop\bk_2} \ \middle| \ {\{1\}^{h+1}\atop\{1\}^{h+1}}\right).
\]
Let $\bk_1, \bk_2$ be indices, $\bz_1\in\DD^{\dep(\bk_1)}$, $\bz_2\in\DD^{\dep(\bk_2)}$ and $v\in\BB\coloneqq\BB_1(1)\setminus\{\infty\}$.
We assume that if $\bk_1=\emp$ (resp.~$\bk_2=\emp$), then $\mathrm{Re}(v)\neq\frac{1}{2}$ (resp.~$|v|\neq 1$).
Then, by transport relations for $Z_2$ (Example~\ref{ex:Z_2_gen} \eqref{it:Ex3.1-1}, \eqref{it:Ex3.1-2}), the following relations hold:
\begin{align*}
Z^{(h)}\left(\left({\bz_1\atop\bk_1}\right)_{\overset{v}{\longrightarrow}};{\bz_2\atop\bk_2}\right)&=-Z^{(h)}\left({\bz_1\atop\bk_1};\left({\bz_2\atop\bk_2}\right)_{\overset{\frac{v}{v-1}}{\longrightarrow}}\right)-Z^{(h-1)}\left(\left({\bz_1\atop\bk_1}\right)_{\overset{v}{\longrightarrow}};\left({\bz_2\atop\bk_2}\right)_{\overset{\frac{v}{v-1}}{\longrightarrow}}\right) \ (v\neq 1),\\
Z^{(h)}\left(\left({\bz_1\atop\bk_1}\right)_{\overset{1}{\longrightarrow}};{\bz_2\atop\bk_2}\right)&=Z^{(h)}\left({\bz_1\atop\bk_1};\left({\bz_2\atop\bk_2}\right)_{\up}\right)+Z^{(h-1)}\left(\left({\bz_1\atop\bk_1}\right)_{\overset{1}{\longrightarrow}};\left({\bz_2\atop\bk_2}\right)_{\up}\right)\quad (\bk_2\neq\emp),\\
Z^{(h)}\left(\left({\bz_1\atop\bk_1}\right)_{\up};{\bz_2\atop\bk_2}\right)&=Z^{(h)}\left({\bz_1\atop\bk_1};\left({\bz_2\atop\bk_2}\right)_{\overset{1}{\longrightarrow}}\right)-Z^{(h-1)}\left(\left({\bz_1\atop\bk_1}\right)_{\up};\left({\bz_2\atop\bk_2}\right)_{\overset{1}{\longrightarrow}}\right)\quad (\bk_1\neq\emp).
\end{align*}
Here, set $Z^{(-1)}(*;*)\coloneqq0$; the case where $h=0$ in the above relations gives the transport relations for the duality (Example~\ref{ex:tranport relation for duality}).

It can be proved by direct calculation that the space of relations obtained by restricting to the use of transport relations for $\{Z^{(h)}\}_{h\geq 0}$ coincides with the space of Ohno's relations for MPLs.
However, in this subsection, we derive Ohno's relations for MPLs more clearly in terms of a formal power series ring over a Hoffman-type algebra.
This also gives a new proof of Ohno's relation for MZVs; compare this with another proof using connected sums by the third author and Yamamoto \cite{Seki-Yamamoto}.

First, we prepare some terminology of the Hoffman-type algebra appropriately.
\begin{definition}
Prepare letters $x$ and $e_z$ for all $z\in\BB$.
Let $\hof$ be a non-commutative polynomial algebra $\hof\coloneqq\QQ\langle x, e_z  \ (z\in\BB)\rangle$.
Furthermore, we define its subalgebras $\hof^1$ and $\hof^0$ by
\begin{align*}
&\hof^1\coloneqq\QQ\oplus\bigoplus_{z\in\mathbb{B}} e_z\hof\supset\\
&\hof^0\coloneqq\QQ\oplus\bigoplus_{\substack{z\in\BB, \\ \mathrm{Re}(z)\neq \frac{1}{2}, |z|\neq 1}}\QQ e_z\oplus\bigoplus_{z\in\mathbb{B}, \ \mathrm{Re}(z)\neq \frac{1}{2}}e_z\hof x\oplus\bigoplus_{\substack{z\in\mathbb{B}, \ \mathrm{Re}(z)\neq \frac{1}{2}\\z^{\prime}\in\mathbb{B}, \ |z^{\prime}|\neq 1}} e_z\hof e_{z^{\prime}}.
\end{align*}
We say that a word $w\in\hof^1$ is associated with $\left({\bz\atop\bk}\right)$ when $w=e_{z_1}x^{k_1-1}\cdots e_{z_r}x^{k_r-1}$, where $\bz=(z_1,\dots,z_r)$ ($z_i\in\BB$) and $\bk=(k_1,\dots,k_r)$ ($k_i\in\ZZ_{>0}$).
In addition, $1\in\hof^1$ is associated with $\left({\emp\atop\emp}\right)$.
Note that a word of $\hof^0$ is associated with a pair satisfying the dual condition in \S\ref{subsec:duality for MPL}.

We define a $\QQ$-linear map $L\colon\hof^0\jump{t}\to\CC\jump{t}$ by 
\[
L\left(\sum_{i=0}^{\infty}w_it^i\right)\coloneqq \sum_{i=0}^{\infty}L(w_i)t^i
\]
and a $\QQ$-linear map $\mathfrak{L}\colon\hof^0\jump{t}\to\CC\jump{t}$ by
\[
\mathfrak{L}\left(\sum_{i=0}^{\infty}w_it^i\right)\coloneqq\sum_{i=0}^{\infty}\mathfrak{L}(w_i)t^i.
\] 
Here, 
\[
L(w)\coloneqq \Li_{\bk}(\bz),\quad \mathfrak{L}(w)\coloneqq \sum_{h=0}^{\infty}Z^{(h)}\left({\bz\atop\bk};{\emp\atop\emp}\right)t^h
\]
for each word $w\in\hof^0$ associated with $\left({\bz\atop\bk}\right)$.
\end{definition}
\begin{definition}
Let $\hof\jump{t}$ be the formal power series ring over $\hof$.
We define two automorphisms $\sigma, \rho$ and two anti-automorphisms $\tau,\tau'$ on $\hof\jump{t}$ by
\[
\begin{array}{lll}
\sigma(x)=x,\quad & &\sigma(e_z)=e_z(1-xt)^{-1}\quad (z\in\BB),\\
\rho(x)=x,& &\rho(e_z)=e_z(1-e_zt)^{-1}\quad (z\in\BB), \\
\tau(x)=e_1, &\tau(e_1)=x,\quad & \tau(e_z)=-e_{\frac{z}{z-1}}(1-e_{\frac{z}{z-1}}t)^{-1} \quad (z\in \BB\setminus\{1\}),\\
\tau'(x)=e_1(1+e_1t)^{-1}, &\tau'(e_1)=x(1-xt)^{-1}, &\tau'(e_z)=-e_{\frac{z}{z-1}}(1+e_{\frac{z}{z-1}}t)^{-1} \quad (z\in \BB\setminus\{1\})
\end{array}
\]
and $\sigma(t)=\rho(t)=\tau(t)=\tau'(t)=t$.
Images of the generators under $\sigma^{-1}$ and $\rho^{-1}$ are given by
\[
\begin{array}{ll}
\sigma^{-1}(x)=x,\quad &\sigma^{-1}(e_z)=e_z(1-xt)\quad (z\in\BB),\\
\rho^{-1}(x)=x,& \rho^{-1}(e_z)=e_z(1+e_zt)^{-1}\quad (z\in\BB), \\
\end{array}
\]
with $\sigma^{-1}(t)=\rho^{-1}(t)=t$.
Furthermore, $\tau^{-1}=\tau$ and $\tau'^{-1}=\tau'$ hold.
\end{definition}
\begin{lemma}\label{lem:rho_tau}
$\rho\circ\tau'=\tau\circ\rho$.
\end{lemma}
\begin{proof}
It is sufficient to check that $(\rho\circ\tau')(u)=(\tau\circ\rho)(u)$ holds for  $u\in\{x\}\cup\{e_z \mid z\in \BB\}$.
\end{proof}
\begin{proposition}[Boundary condition]\label{prop:Hoffman boundary condtion}
For every $w\in\hof^0\jump{t}$, we have
\[
\mathfrak{L}(w)=L\left((\sigma\circ\rho)(w)\right).
\]
\end{proposition}
\begin{proof}
It is sufficient to show the case $w=e_{z_1}x^{k_1-1}\cdots e_{z_r}x^{k_r-1}\in\hof^0$.
By definition, we have
\begin{align*}
(\sigma\circ\rho)(w)&=\sigma\left(e_{z_1}(1-e_{z_1}t)^{-1}x^{k_1-1}\cdots e_{z_r}(1-e_{z_r}t)^{-1}x^{k_r-1}\right)\\
&=e_{z_1}(1-xt-e_{z_1}t)^{-1}x^{k_1-1}\cdots e_{z_r}(1-xt-e_{z_r}t)^{-1}x^{k_r-1}\\
&=\sum_{h=0}^{\infty} \sum_{\substack{f_1,\ldots,f_r\ge 0\\f_1+\cdots+f_r=h}}\left(e_{z_1}(e_{z_1}+x)^{f_1}x^{k_1-1}\cdots e_{z_r}(e_{z_r}+x)^{f_r}x^{k_r-1}\right)t^h.
\end{align*}
Since 
\begin{align*}
&L\Biggl(\sum_{\substack{f_1,\ldots,f_r\ge 0\\f_1+\cdots+f_r=h}}e_{z_1}(e_{z_1}+x)^{f_1}x^{k_1-1}\cdots e_{z_r}(e_{z_r}+x)^{f_r}x^{k_r-1}\Biggr)\\
&=\sum_{\substack{0<m_1<\cdots<m_r \\ 0=i_0< i_1\leq \cdots \leq i_h\leq m_r}}\frac{z_1^{m_1}z_2^{m_2-m_1}\cdots z_r^{m_r-m_{r-1}}}{m_1^{k_1}\cdots m_r^{k_r}i_1\cdots i_h}=Z^{(h)}\left({z_1,\dots, z_r\atop k_1,\dots, k_r};{\emp\atop\emp}\right)
\end{align*}
holds, we have $L\left((\sigma\circ\rho)(w)\right)=\mathfrak{L}(w)$ by definition.
\end{proof}
\begin{definition}
For a non-negative integer $h$, we define a $\QQ$-bilinear map $\fZ^{(h)}\colon\hof'\times\hof'\to\CC$ by
\[
\fZ^{(h)}(w_1;w_2)\coloneqq Z^{(h)}\left({\bz_1\atop\bk_1};{\bz_2\atop\bk_2}\right),
\]
where $\hof'\coloneqq \bigoplus_{z\in\BB}e_z\hof$ (non-unital) and $w_1$ and $w_2$ are associated with $\bigl({\bz_1\atop\bk_1}\bigr)$ and $\bigl({\bz_2\atop\bk_2}\bigr)$, respectively.
Furthermore, we define a $\QQ$-bilinear map $\fZ\colon\hof'\jump{t}\times\hof'\jump{t}\to\CC\jump{t}$ by
\[
\fZ\left(\sum_{i=0}^{\infty}w_it^i;\sum_{j=0}^{\infty}w'_jt^j\right)\coloneqq \sum_{i=0}^{\infty}\sum_{j=0}^{\infty}\fZ(w_i;w'_j)t^{i+j}
\]
and
\[
\fZ(w_i;w'_j)\coloneqq\sum_{h=0}^{\infty}\fZ^{(h)}(w_i;w'_j)t^h.
\]
\end{definition}
\begin{proposition}[Transport relation]\label{prop:Hoffman transport relations}
We have the following:
\begin{enumerate}[\rm(i)]
\item For $w_1, w_2 \in \hof'\jump{t}$ and $u\in \hof\jump{t}$,
\[
\fZ(w_1u;w_2)=\fZ(w_1;w_2\tau'(u)).
\]
\item For $w\in\hof'\jump{t}$ and $u\in\{x\}\cup\{e_z\mid z\in\BB, |z|\neq 1\}$,
\[
\mathfrak{L}(wu)=\fZ(w;\tau'(u)).
\]
\item For $w\in\hof'\jump{t}$ and $u\in\{e_z\mid z\in \BB, \mathrm{Re}(z)\neq \frac{1}{2}\}$,
\[
\fZ(u;w)=\mathfrak{L}(w\tau'(u)).
\]
\item For $z\in\BB$ with $\mathrm{Re}(z)\neq \frac{1}{2}$ and $|z|\neq 1$,
\[
\mathfrak{L}(e_z)=\mathfrak{L}(\tau'(e_z)).
\]
\end{enumerate}
\end{proposition}
\begin{proof}
It is sufficient to show the case where $w_1,w_2,w$ are words and $u$ is a letter.
In this case, it follows immediately from the transport relations for $\{Z^{(h)}\}_{h\geq 0}$.
\end{proof}
\begin{theorem}\label{thm:Hoffman Ohno}
For $w\in\hof^0\jump{t}$, we have
\[
L(\sigma(w))=L(\sigma(\tau(w))).
\]
\end{theorem}
\begin{proof}
By Proposition~\ref{prop:Hoffman transport relations}, we have
\[
\mathfrak{L}(w)=\mathfrak{L}(\tau'(w)).
\]
Therefore, by Proposition~\ref{prop:Hoffman boundary condtion} and Lemma~\ref{lem:rho_tau}, we have
\[
L\left(\sigma(\rho(w))\right)=L\left(\sigma((\rho\circ\tau')(w))\right)=L\left((\sigma\circ\tau)(\rho(w))\right).
\]
Since $\rho$ is invertible and $\rho^{-1}(\hof^0\jump{t})\subset\hof^0\jump{t}$, we have $L(\sigma(w))=L(\sigma(\tau(w)))$ for all $w\in\hof^0\jump{t}$.
\end{proof}
To restate Theorem~\ref{thm:Hoffman Ohno} without terminology of the Hoffman-type algebra, we define some new notation.
Recall the notion of the dual condition, $\left({\bz\atop\bk}\right)^{\dagger}$ and $\iota(\bz)$ from \S\ref{subsec:duality for MPL}.
\begin{definition}
Let $r$ be a positive integer, $\bk=(k_1,\dots, k_r)$ an index and $\bz=(z_1,\dots,z_r)\in\DD^r$ satisfying that if $k_r=1$, then $|z_r|\neq 1$.
For a non-negative integer $h$, we define $O_h\left({\bz\atop\bk}\right)$ by
\[
O_h\left({\bz\atop\bk}\right)\coloneqq \sum_{\substack{c_1,\dots, c_r\geq 0 \\ c_1+\cdots +c_r=h}}\Li_{k_1+c_1, \dots, k_r+c_r}(\bz).
\]
\end{definition}
\begin{definition}
Let $\bk$ be a non-empty index and $\bz\in \DD^{\dep(\bk)}$.
Set $d=\iota(\bz)$.
When $\left({\bz\atop\bk}\right)$ is written as
\[
\left({\bz\atop\bk}\right)=\left({\{1\}^{a_1},\atop\bk_1,}{z_1\atop l_1}{,\dots,\atop,\dots,}{\{1\}^{a_d},\atop\bk_d,}{z_d,\atop l_d,}{\{1\}^{a_{d+1}}\atop\bk_{d+1}}\right),
\]
where $a_1,\dots, a_{d+1}$ are non-negative integers, $z_1,\dots, z_d\in \DD\setminus\{1\}$, $\bk_1,\dots, \bk_{d+1}$ are indices with $\dep(\bk_1)=a_1,\dots, \dep(\bk_{d+1})=a_{d+1}$ and $l_1,\dots,l_d$ are positive integers.
Let $b_1,\dots, b_d$ be non-negative integers.
Then we define $\left({\bz\atop\bk}\right)_{b_1,\dots, b_d}$ by
\[
\left({\bz\atop\bk}\right)_{b_1,\dots, b_d}\coloneqq
\left({\{1\}^{a_1},\atop\bk_1,}{\{z_1\}^{b_1+1}\atop\{1\}^{b_1},l_1}{,\dots,\atop,\dots,}{\{1\}^{a_d},\atop\bk_d,}{\{z_d\}^{b_d+1},\atop\{1\}^{b_d},l_d,}{\{1\}^{a_{d+1}}\atop\bk_{d+1}}\right).
\]
\end{definition}
\begin{theorem}[Ohno's relation for multiple polylogarithms]\label{thm:Ohno MPLs}
Let $\bk$ be a non-empty index and $\bz\in\DD^{\dep(\bk)}$.
We assume that the pair $\left({\bz\atop\bk}\right)$ satisfies the dual condition and write $\left({\bz\atop\bk}\right)^{\dagger}$ as $\bigl({\bz'\atop\bk'}\bigr)$.
Set $d=\iota(\bz)$.
Then, for any non-negative integer $h$, we have
\[
O_h\left({\bz\atop\bk}\right)=(-1)^{d}\sum_{i=0}^h\sum_{\substack{b_1,\dots,b_d\geq 0 \\ b_1+\dots +b_d=i}}O_{h-i}\left(\left({\bz'\atop\bk'}\right)_{b_1,\dots,b_d}\right).
\]
In particular, when $d=0$ $($then $\bk$ is admissible$)$, this means
\[
O_h\left({\{1\}^{\dep(\bk)}\atop\bk}\right)=O_h\left({\{1\}^{\dep(\bk^{\dagger})}\atop\bk^{\dagger}}\right),
\]
which is the well-known Ohno's relation for multiple zeta values \cite{Ohno}.
\end{theorem}
\begin{proof}
In Theorem~\ref{thm:Hoffman Ohno}, we compare the coefficients of $t^h$ for the case where $w\in \hof^0$ is a word.
Then we have the desired formula from various definitions of notation. 
\end{proof}
In the following corollary, we use the symbol $\mathrm{Li}_{\bk}(z)\coloneqq\mathrm{Li}^*_{\bk}(\{1\}^{\dep(\bk)-1},z)=\Li_{\bk}(\{z\}^{\dep(\bk)})$ for the $1$-variable multiple polylogarithm.
\begin{corollary}[The depth $1$ case of the Landen connection formula, \cite{Okuda-Ueno}]
For a positive integer $k$ and a complex number $z$ satisfying $|z|<1$ and $\mathrm{Re}(z)<\frac{1}{2}$, we have
\[
\mathrm{Li}_k(z)=-\sum_{\bk \ \text{with } \wt(\bk)=k}\mathrm{Li}_{\bk}\left(\frac{z}{z-1}\right),
\]
where $\bk$ runs over all indices whose weight equal $k$.
\end{corollary}
\begin{proof}
This is the case where $\left({\bz \atop \bk}\right) = \left({z\atop 1}\right)$ and $h=k-1$ in Theorem~\ref{thm:Ohno MPLs}.
\end{proof}
\subsection{Other relations}
\begin{example}\label{ex:Dilcher}
For $k\geq 2$, we have
\begin{equation}\label{eq:Dilcher-Vignat}
\zeta(k)=(-1)^{k-1}\zeta(\{1\}^{k-2},\overline{2})+\sum_{j=2}^{k}(-1)^{k-j+1}\zeta(\{1\}^{k-j},\overline{j}).
\end{equation}
\end{example}
\begin{proof}
Consider the data $n=2$, $\bk_1=(\{1\}^{k-1})$, $\bz_1=(\{1\}^{k-1})$, $\bl=(2)$ and $\bw=(1)$ in the recipe in \S\ref{subsec:Recipe for relations}.
First, with the duality for MZVs, we have
\[
Z_2\left({\{1\}^{k-1}\atop\{1\}^{k-1}};{\emp\atop\emp} \ \middle| \ {1\atop 2}\right)=\zeta(\{1\}^{k-2},2)=\zeta(k).
\]
On the other hand, by Example~\ref{ex:Z_2_gen} \eqref{it:Ex3.1-3}, 
\begin{align*}
&Z_2\left({\{1\}^{k-1}\atop\{1\}^{k-1}};{\emp\atop\emp} \ \middle| \ {1\atop 2}\right)\\
&=-Z_2\left({\{1\}^{k-2}\atop\{1\}^{k-2}};{-1\atop\phantom{-}1} \ \middle| \ {1\atop 2}\right)-Z_2\left({\{1\}^{k-1}\atop\{1\}^{k-1}};{-1\atop\phantom{-}1}\right)\\
&=\cdots\\
&=(-1)^{k-1}Z_2\left({\emp\atop\emp};{\{-1\}^{k-1}\atop \{1\}^{k-1}} \ \middle| \ {1\atop 2}\right)+\sum_{j=2}^k(-1)^{k-j+1}Z_2\left({\{1\}^{j-1}\atop\{1\}^{j-1}};{\{-1\}^{k-j+1}\atop\{1\}^{k-j+1}}\right)
\end{align*}
holds, and by the transport relation for the duality (Example~\ref{ex:tranport relation for duality}),
\[
Z_2\left({\{1\}^{j-1}\atop\{1\}^{j-1}};{\{-1\}^{k-j+1}\atop\{1\}^{k-j+1}}\right)=\cdots=\Li_{\{1\}^{k-j},j}(\{-1\}^{k-j+1})=\zeta(\{1\}^{k-j},\overline{j})
\]
holds for each $j$.
Furthermore, 
\[
Z_2\left({\emp\atop\emp};{\{-1\}^{k-1}\atop \{1\}^{k-1}} \ \middle| \ {1\atop 2}\right)=\Li_{\{1\}^{k-2},2}(\{-1\}^{k-1})=\zeta(\{1\}^{k-2},\overline{2})
\]
holds and thus we have the desired formula.
\end{proof}
\begin{remark}
By using $\zeta(\overline{k})=-(1-2^{1-k})\zeta(k)$, equality \eqref{eq:Dilcher-Vignat} is rewritten as
\[
\zeta(k)=(-2)^{k-1}\left(2\zeta(\{1\}^{k-2},\overline{2})+\sum_{j=3}^{k-1}(-1)^j\zeta(\{1\}^{k-j},\overline{j})\right),
\]
which was recently proved by Dilcher--Vignat \cite[Theorem~4.1]{Dilcher-Vignat}; their proof is different from ours.
The case where $k=3$ is the famous formula $\zeta(3)=8\zeta(1,\overline{2})$; see \cite{Borwein-Bradley}.
\end{remark}
\begin{example}\label{ex:Li11}
Let $z_1,z_2,z_3$ be elements of $\{z\in\CC \mid 0<|z|<1, \mathrm{Re}(z)<\frac{1}{2}\}$ satisfying $\frac{1}{z_1}+\frac{1}{z_2}+\frac{1}{z_3}=1$.
Then we have
\[
\Li_{1,1}\left(z_1,\frac{z_2}{z_2-1}\right)+\Li_{1,1}\left(z_2,\frac{z_3}{z_3-1}\right)+\Li_{1,1}\left(z_3,\frac{z_1}{z_1-1}\right)=0.
\]
\end{example}
\begin{proof}
By the fundamental identity, we have
\[
Z_3\left({z_1\atop 1};{z_2\atop 1};{\emp\atop\emp}\right)+Z_3\left({z_1\atop 1};{\emp\atop\emp};{z_3\atop 1}\right)+Z_3\left({\emp\atop\emp};{z_2\atop 1};{z_3\atop 1}\right)=0
\]
and by Propositions~\ref{prop:fundamental_properties} \eqref{item:reduction}, this is
\begin{equation}\label{eq:Z_2 triple}
Z_2\left({z_1\atop 1};{z_2\atop 1}\right)+Z_2\left({z_1\atop 1};{z_3\atop 1}\right)+Z_2\left({z_2\atop 1};{z_3\atop 1}\right)=0.
\end{equation}
Then we have the desired formula by a similar calculation as in equality~\eqref{eq:Z11}.
\end{proof}
\begin{remark}
By applying Example~\ref{ex:var Cloitre} to \eqref{eq:Z_2 triple}, we have the six-term relation for dilogarithms ($z_1,z_2,z_3$ as in Example~\ref{ex:Li11})
\[
\mathrm{Li}_2(z_1)+\mathrm{Li}_2(z_2)+\mathrm{Li}_2(z_3)=\frac{1}{2}\left\{\mathrm{Li}_2\left(-\frac{z_1z_2}{z_3}\right)+\mathrm{Li}_2\left(-\frac{z_2z_3}{z_1}\right)+\mathrm{Li}_2\left(-\frac{z_3z_1}{z_2}\right)\right\}
\]
attributable to Kummer and Newman; see \cite[p.9]{Zagier}.
\end{remark}
\begin{example}
Let $z_1,z_2,z_3,z_4$ be elements of $\{z\in\CC \mid 0<|z|<1, \mathrm{Re}(z)<\frac{1}{2}\}$ satisfying $\frac{1}{z_1}+\frac{1}{z_2}+\frac{1}{z_3}+\frac{1}{z_4}=1$ and $|g(a,b)|\leq 1$ for $(a,b)=(z_1,z_2), (z_2,z_3), (z_3,z_4), (z_4,z_1)$, where $g(a,b)\coloneqq \frac{ab}{ab-a-b}$.
Then we have an eight-term relation for $\Li_{1,1,1}$,
\[
\mathcal{L}(z_1,z_2,z_3)+\mathcal{L}(z_2,z_3,z_4)+\mathcal{L}(z_3,z_4,z_1)+\mathcal{L}(z_4,z_1,z_2)=0,
\]
where
\[
\mathcal{L}(a,b,c)\coloneqq \Li_{1,1,1}\left(c,g(a,b),\frac{a}{a-1}\right)+\Li_{1,1,1}\left(c,g(a,b),\frac{b}{b-1}\right).
\]
\end{example}
\begin{proof}
By the fundamental identity, we have
\[
Z_4\left({z_1\atop 1};{z_2\atop 1};{z_3\atop 1};{\emp\atop\emp}\right)+Z_4\left({z_1\atop 1};{z_2\atop 1};{\emp\atop\emp};{z_4\atop 1}\right)+Z_4\left({z_1\atop 1};{\emp\atop\emp};{z_3\atop 1};{z_4\atop 1}\right)+Z_4\left({\emp\atop\emp};{z_2\atop 1};{z_3\atop 1};{z_4\atop 1}\right)=0
\]
and by Proposition~\ref{prop:fundamental_properties} \eqref{item:reduction}, this is
\[
Z_3\left({z_1\atop 1};{z_2\atop 1};{z_3\atop 1}\right)+Z_3\left({z_1\atop 1};{z_2\atop 1};{z_4\atop 1}\right)+Z_3\left({z_1\atop 1};{z_3\atop 1};{z_4\atop 1}\right)+Z_3\left({z_2\atop 1};{z_3\atop 1};{z_4\atop 1}\right)=0.
\]
According to the recipe for evaluations (Theorem~\ref{thm:main1 precise}), we have
\begin{align*}
Z_3\left({z_1\atop 1};{z_2\atop 1};{z_3\atop 1}\right)&=-Z_3\left({\emp\atop\emp};{z_2\atop 1};{z_3,\atop 1,} {g(z_1,z_2)\atop 1}\right)-Z_3\left({z_1\atop 1};{\emp\atop\emp};{z_3,\atop 1,} {g(z_1,z_2)\atop 1}\right)\\
&=\Li_{1,1,1}\left(z_3,g(z_1,z_2),\frac{z_2}{z_2-1}\right)+\Li_{1,1,1}\left(z_3,g(z_1,z_2),\frac{z_1}{z_1-1}\right)\\
&=\mathcal{L}(z_1,z_2,z_3).
\end{align*}
The remaining three $Z_3$-values can be calculated in the same way.
\end{proof}
\subsection*{
Conflict of interest
}
On behalf of all authors, the corresponding author states that there is no conflict of interest. 


\end{document}